\documentclass[11pt,oneside,reqno]{amsart}
\usepackage{amssymb}
\usepackage{amsmath}
\usepackage{amscd}
\usepackage{amsfonts}
\usepackage{mathrsfs}
\usepackage{enumitem}
\usepackage{stmaryrd}
\usepackage{tikz}
\usepackage{pgf,tikz,pgfplots}
\pgfplotsset{compat=1.14}
\definecolor{ffffff}{rgb}{1,1,1}
\definecolor{ffff00}{rgb}{1,1,0}

\usetikzlibrary{arrows}
\usepackage[T1]{fontenc}
\usepackage[normalem]{ulem}
\allowdisplaybreaks

\usepackage{bbm}

\theoremstyle{definition}
\newtheorem{theorem}{Theorem}[section]

\newtheorem{thm}[theorem]{Theorem}
\newtheorem{prop}[theorem]{Proposition}

\newtheorem{defn}[theorem]{Definition}

\newtheorem{cor}[theorem]{Corollary}
\newtheorem{prop-def}{Proposition-Definition}[section]

\newtheorem{rema}[theorem]{Remark}

\newtheorem{nota}[theorem]{Notation}

\newcommand{\N}{{\mathbb N}}
\newcommand{\C}{{\mathbb C}}
\newcommand{\Z}{{\mathbb Z}}

\newcommand{\h}{{\mathfrak h}}

\newcommand{\End}{\textrm{End}}

\newcommand{\one}{\mathbf{1}}
\renewcommand{\d}{\mathbf{d}}
\newcommand{\wt}{\mbox{\rm wt}\ }

\newcommand{\Sym}{\text{Sym}}

\usepackage{scalerel,stackengine}
\stackMath
\newcommand\reallywidehat[1]{%
\savestack{\tmpbox}{\stretchto{%
  \scaleto{%
    \scalerel*[\widthof{\ensuremath{#1}}]{\kern-.6pt\bigwedge\kern-.6pt}%
    {\rule[-\textheight/2]{1ex}{\textheight}}
  }{\textheight}%
}{0.5ex}}%
\stackon[1pt]{#1}{\tmpbox}%
}
\begin{document}

\setlength{\oddsidemargin}{0cm} \setlength{\evensidemargin}{0cm}
\baselineskip=18pt

\title[Fermionic construction of the $\frac \Z 2$-graded MOSVA and its $\Z_2$-twisted module, I ]{Fermionic construction of the $\frac \Z 2$-graded meromorphic open-string vertex algebra and its $\Z_2$-twisted module, I}
\author{Francesco Fiodalisi, Fei Qi}

\begin{abstract}
We define the $\frac \Z 2$-graded meromorphic open-string vertex algebra that is an appropriate noncommutative generalization of the vertex operator superalgebra. We also illustrate an example that can be viewed as a noncommutative generalization of the free fermion vertex operator superalgebra. The example is bulit upon a universal half-integer-graded non-anti-commutative Fock space where a creation operator and an annihilation operator satisfy the fermionic anti-commutativity relation, while no relations exist among the creation operators. The former feature allows us to define the normal ordering, while the latter feature allows us to describe interactions among the fermions. With respect to the normal ordering, Wick's theorem holds and leads to a proof of weak associativity and a closed formula of correlation functions. 
\end{abstract}
\maketitle

\section{Introduction}

Vertex (operator) algebras are algebraic structures formed by (meromorphic) vertex operators. In mathematics, they arose naturally in the study of representations of infinite-dimensional Lie algebras and the Monster group (see \cite{FLM} and \cite{B}). In physics, they arose in the study of two-dimensional conformal field theory (see \cite{BPZ} and \cite{MS}). Vertex (operator) superalgebras are similar structures formed by (meromorphic) super vertex operators with parity. They appeared along with vertex operator algebras (see \cite{B}, \cite{KW}, \cite{T}). In mathematics, they arose in the study of spinor representations of affine Lie algebras of type $D$ and $E$ (see \cite{FFR}). In physics, they arose in the study of superconformal field theory (see \cite{NS}, \cite{DPZ}). It should be noted that these examples of vertex operator superalgebras are all ``free field realizations'' in the sense that they are built upon a Fock space with free fermions that do not interact with each other. The resulting vertex operator satisfies both super-commutativity and associativity. 

Meromorphic open-string vertex algebras (MOSVA hereafter) are algebraic structures formed by meromorphic vertex operators that satisfy associativity but do not satisfy commutativity of any kind. It was introduced by Huang in 2012 (see \cite{H-MOSVA}). Physically, associativity can be viewed as a strong form of the operator product expansion (OPE hereafter) of meromorphic fields, which is expected to hold for all quantum field theories. Commutativity, on the other hand, is not expected to hold in general quantum field theories, e.g., in a quantum two-dimensional nonlinear $\sigma$-model with a nonflat target manifold. 

Huang also constructed the MOSVA over a Riemannian manifold (see \cite{H-MOSVA-Riemann}) using parallel tensors and natural modules generated by eigenfunctions of the Laplacian operator. In physics, the eigenfunctions correspond to the quantum states of a particle, which can be viewed as a degenerated form of a string. Elements of the MOSVA modules generated by eigenfunctions can be viewed as suitable string-theoretic excitations of the particle states. It is Huang's idea that the MOSVAs constructed from Riemannian manifolds, together with modules generated by Laplacian eigenfunctions and the still-to-be-defined intertwining operators among these modules may lead to a mathematical construction of the quantum two-dimensional nonlinear $\sigma$-model. 

Huang's idea is motivated by the case when the Riemannian manifold is a torus with zero sectional curvature. In this case, the MOSVA is precisely the usual Heisenberg VOA. The direct sum of eigenfunction modules and the intertwining operators among them form the lattice intertwining operator algebra. In particular, if we only consider the eigenfunctions and the intertwining operators with meromorphic correlation functions, we obtain the lattice VOA that played a key role in the construction of the Moonshine module for the Monster group. The corresponding 2d nonlinear $\sigma$-model is precisely a conformal field theory given by the lattice intertwining operator algebra. 

As the first step of carrying out Huang's program, in \cite{Q-2d-space-form}, the second author computed the structure of MOSVA and its eigenfunction modules over any 2-dimensional space form, i.e., a 2-dimensional complete Riemannian manifold with constant sectional curvature. Most results in \cite{Q-2d-space-form} are expected to hold in higher-dimensional space forms. However, the proof is obstructed by a geometric and combinatorial conjecture regarding the zero-mode actions. See \cite{Q-Cov} for more details. 

It should be noted that the MOSVAs and modules considered in \cite{H-MOSVA}, \cite{H-MOSVA-Riemann}, and \cite{Q-2d-space-form} are all bosonic constructions. Naturally, one would expect a fermionic construction and similar geometric constructions. The fermionic construction is expected to appear in a quantum 2d nonlinear $\sigma$-model with Calabi-Yau target manifolds. 

This paper finishes the first step of fermionic construction. We found that a $\frac \Z 2$-graded MOSVA is an appropriate generalization of vertex operator superalgebras. We also construct an example of such algebraic structure that should be viewed as non-anti-commutative generalizations of the structures discussed in \cite{T} and \cite{FFR}, as well as a Bosonic version of MOSVA in \cite{H-MOSVA}. 

In more detail, the structure is built upon a non-anti-commutative Fock space where the creation operators do not have any relations. Physically, the fermions in the state space can have interactions. Depending on the context, one may require various types of interactions, which leads to various quotients of the Fock space discussed in this paper. To this sense, the Fock space is ``universal.'' We still name the construction ``fermionic'' because we still impose the anti-commutativity relation between a creation operator and an annihilation operator. Such a relation allows us to define the fermionic normal ordering, though interchanging two creation operators is not allowed. We believe that this universal non-anti-commutative Fock space is useful when modeling interacting fermions in the future, which amounts to imposing relations among creation operators and shall result in a certain quotient. 

The verification of the axioms of the $\frac \Z 2$-graded MOSVA essentially depends on an explicit formula of the products of two normal-ordered products (see Theorem \ref{Product-thm}). The formula can be viewed as a generalization of Wick's theorem that expresses the products of normal-ordered products in terms of other normal-ordered products. It should be noted that though the same formula holds implicitly in the vertex operator superalgebra constructed in \cite{FFR} and in \cite{T}, since the state space is larger, the correlation functions shall not be identical. 

One of the lessons learned from \cite{Q-2d-space-form} is that the MOSVA itself might carry little geometric and physical information. For that purpose, it is still necessary to construct the canonically $\Z_2$-twisted modules with respect to the parity automorphism. From considerations on the length and notational convenience, we will leave the discussion of canonically $\Z_2$-twisted module in \cite{Q-Fermion-2}, namely Part II of the paper. We shall, however, include in this part the discussions of the $\exp(\Delta(x))$-operator to be used in Part II. 

The paper is organized as follows: Section 2 discusses the basic properties of $\frac \Z 2$-graded MOSVA. Section 3 discusses the non-anti-commutative Fock space, the fermionic normal-ordering, and the definition of vertex operators. We check all the axioms except for the weak associativity. Section 4 proves a generalization of Wick's theorem, which leads to the proof of weak associativity. Section 5 discusses the operator $\Delta(x)$ and its exponential to prepare for Part II. 

\noindent \textbf{Acknowledgements. }The authors thank Yi-Zhi Huang for suggesting the problem and for his long term support. 


\section{$\frac \Z 2$-graded MOSVA}

\begin{defn}\label{DefMOSVA}
{\rm A {\it $\frac \Z 2$-graded meromorphic open-string vertex algebra} (hereafter MOSVA) is a $\frac \Z 2$-graded vector space 
$V=\coprod\limits_{n\in\frac \Z 2} V_{(n)}$ (graded by {\it weights}) equipped with a {\it vertex operator map}
\begin{eqnarray*}
   Y_V:  V\otimes V &\to & V[[x,x^{-1}]]\\
	u\otimes v &\mapsto& Y_V(u,x)v,
  \end{eqnarray*}
and a {\it vacuum} $\one\in V$, satisfying the following axioms:
\begin{enumerate}
\item Axioms for the grading:
\begin{enumerate}
\item {\it Lower bound condition}: When $n$ is sufficiently negative,
$V_{(n)}=0$.
\item {\it $\d$-commutator formula}: Let $\d_{V}: V\to V$
be defined by $\d_{V}v=nv$ for $v\in V_{(n)}$. Then for every $v\in V$
$$[\d_{V}, Y_{V}(v, x)]=x\frac{d}{dx}Y_{V}(v, x)+Y_{V}(\d_{V}v, x).$$
\end{enumerate}

\item Axioms for the vacuum: 
\begin{enumerate}
\item {\it Identity property}: Let $1_{V}$ be the identity operator on $V$. Then
$Y_{V}(\mathbf{1}, x)=1_{V}$. 
\item {\it Creation property}: For $u\in V$, $Y_{V}(u, x)\mathbf{1}\in V[[x]]$ and 
$\lim_{x\to 0}Y_{V}(u, x)\mathbf{1}=u$.
\end{enumerate}

\item {\it $D$-derivative property and $D$-commutator formula}:
Let $D_V: V\to V$ be the operator
given by
$$D_{V}v=\lim_{x\to 0}\frac{d}{dx}Y_{V}(v, x)\one$$
for $v\in V$. Then for $v\in V$,
$$\frac{d}{dx}Y_{V}(v, x)=Y_{V}(D_{V}v, x)=[D_{V}, Y_{V}(v, x)].$$

\item {\it Weak associativity with pole-order condition}: For every $u_1, v\in V$, there exists $p\in \mathbb{N}$ such that for every $u_2\in V$, 
$$(x_0+x_2)^p Y_V(u_1, x_0+x_2)Y_V(u_2, x_2)v = (x_0+x_2)^p Y_V(Y_V(u_1, x_0)u_2, x_2)v.$$
\end{enumerate}  
}
\end{defn}

\begin{rema}
A grading-restricted vertex superalgebra in the sense of \cite{H-Two-Constructions} satisfies all the properties above. So does a vertex operator superalgebra in the sense of \cite{KW}. Thus, the notion of $\frac \Z 2$-graded meromorphic open-string vertex algebra is an appropriate noncommutative generalization of the notion of vertex superalgebras. We may use the notation 
$$V_0 = \coprod\limits_{n\in \Z} V_{(n)}, V_1 = \coprod\limits_{n\in \Z+\frac 1 2} V_{(n)}$$
and call $V_0$ the \textit{even} part, $V_1$ the \textit{odd} part of the MOSVA. Clearly, the map
$$\theta: V \to V, \theta(v)=(-1)^i v, v\in V_i, i=1, 2$$
is an involution of the MOSVA, i.e., 
$$\theta(Y_V(u,x)v)=Y_V(\theta(u), x)\theta(v), \theta^2(v) = v, u, v\in V. $$
We call $\theta$ the \textit{parity involution}. 
\end{rema}

\begin{rema}
    Theorem 3.4 in \cite{Q-Mod} can be generalized verbatim to a $\frac \Z 2$-graded MOSVA, showing that the weak associativity assumed in Definition \ref{DefMOSVA} is sufficient to guarantee the convergence of products of any number of vertex operators, i.e., for every $v'\in V' = \coprod\limits_{n\in \frac \Z 2} V_{(n)}^*, u_1, ..., u_r, v\in V$, the complex series
    $$\langle v', Y_V(u_1, z_1)\cdots Y_V(u_r, z_r)v\rangle$$
    converges absolutely in the region 
    $$|z_1|>\cdots > |z_r|>0$$
    to a rational function with the only possible poles at $z_i = 0$ $(i = 1, ..., r)$ and at $z_i = z_j$ $(1\leq i < j \leq r)$. Same applies to Theorem 4.10 in \cite{Q-Mod} that conclude the convergence of iterates of any number of vertex operators, i.e., for every $v'\in V' = \coprod\limits_{n\in \frac \Z 2} V_{(n)}^*, u_1, ..., u_r, v\in V$, the complex series
    $$\langle v', Y_V(Y_V(\cdots Y_V(Y_V(u_1, z_1-z_2)u_2, z_2-z_3)u_3, \cdots)u_r, z_r)v\rangle$$
    converges absolutely in the region 
    \begin{equation*}
        \left\{(z_1, ..., z_n)\in \C^n: \begin{aligned}
        &|z_n|>|z_i-z_n|>0, i = 1, ..., n; \\
        &|z_i-z_{i+1}|>|z_j-z_i|>0, 1\leq j < i \leq n-1 \end{aligned}\right\}
    \end{equation*}
    to the same rational function the product converges to. The convergence of mixed products and iterates can be similarly proved.  
\end{rema}

\section{Fermionic construction of vertex operators}


\subsection{The algebra $N(\hat{h}_{\Z+1/2})$ and its induced modules $V$}

Let $\h = \C^{2M}$ be a vector space with a nondegenerate symmetric bilinear form $(\cdot, \cdot)$, such that 
$$\h = \mathfrak{p}\oplus \mathfrak{q}$$ 
is a polarization into maximal isotropic subspaces $\mathfrak{p}, \mathfrak{q} \simeq \C^{M}$. 
We consider the half-integral affinization 
$$\hat{\h}_{\Z+1/2} = \h \otimes t^{1/2} \C[t, t^{-1}] \oplus \C \mathbf{k}$$
As a vector space, $\hat{\h}_{\Z+1/2} = \hat{\h}_{\Z+1/2}^+ \oplus \hat{\h}_{\Z+1/2}^- \oplus \hat{\h}_{\Z+1/2}^0$, where 
$$\hat{\h}_{\Z+1/2}^\pm = t^{\pm 1/2}\C[t^{\pm 1}], \hat{\h}_{\Z+1/2}^0 = \C \mathbf{k}$$
Let 
$$T(\hat{\h}_{\Z+1/2}) = \C \oplus  \hat{\h}_{\Z+1/2} \oplus (\hat{\h}_{\Z+1/2})^{\otimes 2} \oplus  \cdots. $$
be the tensor algebra of $\hat{\h}_{\Z+1/2}$. Let $N(\hat{\h}_{\Z+1/2})$ be the quotient of $T(\hat{\h}_{\Z+1/2})$ by the two-sided ideal $J$ generated by 
\begin{eqnarray}
&(a\otimes t^{m+\frac 1 2})\otimes (b\otimes t^{n+\frac 1 2})
+ (b\otimes t^{n+\frac 1 2})\otimes (a\otimes t^{m+ \frac 1 2})\label{ideal-1}
-m(a, b)\delta_{m+n+1, 0}\mathbf{k},&\\
&(a\otimes t^{k+\frac 1 2})\otimes \mathbf{k}-\mathbf{k}\otimes (a\otimes t^{k+\frac 1 2}), & \label{ideal-2}
\end{eqnarray}
for $a, b\in \mathfrak{h}$, $m\in \Z_{+}$, $n\in -\Z_{+}$, $k\in \Z$. 

\begin{rema}
It should be emphasized here that we do not have any relations between $a\otimes t^{m+1/2}$ and $b\otimes t^{n+1/2}$ if $m+1/2,n+1/2$ are both positive or both negative. 
\end{rema}

\begin{prop}\label{PBW}
    As a vector space, $N(\hat{\h}_{\Z+1/2})$ is isomorphic to $T(\hat{\h}_{\Z+1/2}^-)\otimes T(\hat{\h}_{\Z+1/2}^+) \otimes T(\C\mathbf{k})$
    where $T(\hat{\h}_{\Z+1/2}^-)$, $T(\hat{\h}_{\Z+1/2}^+)$, $T(\C\mathbf{k})$ are tensor algebras of $\hat{\h}_{\Z+1/2}^-$, $\hat{\h}_{\Z+1/2}^+$, $\C\mathbf{k}$, respectively. 
\end{prop}

\begin{proof}
    View $T(\hat{\h}_{\Z+1/2}^-)\otimes T(\hat{\h}_{\Z+1/2}^+) \otimes T(\C\mathbf{k})$ as a subspace of $T(\hat{\h}_{\Z+1/2})$. We first show that
    $$T(\hat{\h}_{\Z+1/2}) = T(\hat{\h}_{\Z+1/2}^-)\otimes T(\hat{\h}_{\Z+1/2}^+) \otimes T(\C\mathbf{k}) + J. $$
    Clearly, $T(\hat{\h}_{\Z+1/2})$ is spanned by elements $u_1\otimes \cdots u_n$, where each $u_i$ is either of the form $h\otimes t^{m+1/2}, m\in \Z$ or $\mathbf{k}$. Using (\ref{ideal-2}), we may move all the $\mathbf{k}$ to the right by adding an element in $J$. Using (\ref{ideal-1}), we may arrange the terms $h\otimes t^{m+1/2}$ with $m+1/2<0$ to the left by adding an element in $J$. What remains in the middle are the terms $h\otimes t^{m+1/2}$ with $m+1/2>0$. This way, we see that each $u_1\otimes \cdots \otimes u_n$ in the spanning set can be a sum of an element of the form
    $$(h_1\otimes t^{-m_1-1/2}) \otimes \cdots \otimes (h_k\otimes t^{-m_k-1/2}) \otimes (h_{k+1}\otimes t^{m_{k+1}+1/2}) \otimes \cdots \otimes (h_l\otimes t^{m_l+1/2}) \otimes \mathbf{k} \otimes \cdots \otimes \mathbf{k}$$
    (for some $k,l\in \N, h_1, ..., h_{k+l}\in\h, m_1, ..., m_{k+l}\in \N$), and an element in the ideal $J$. 

    We now show that the sum is indeed direct, i.e., the intersection of $T(\hat{\h}_{\Z+1/2}^-)\otimes T(\hat{\h}_{\Z+1/2}^+) \otimes T(\C\mathbf{k})$ and $J$ is trivial. We follow the same procedure as Jacobson's proof of the linear independence part Poincar\'e-Birkhoff-Witt theorem (see \cite{G-PBW} for an exposition, see also \cite{H-MOSVA}) by constructing a map $L: T(\hat{\h}_{\Z+1/2})\to T(\hat{\h}_{\Z+1/2})$ such that 
    $$L(J) = 0, L|_{T(\hat{\h}_{\Z+1/2}^-)\otimes T(\hat{\h}_{\Z+1/2}^+) \otimes T(\C\mathbf{k})} = 1. $$
    Then every $v$ in the intersection must satisfy both $L(v) = 0$ and $L(v) = v$, which implies that $v=0$. 

    Consider a basis $e_1, ..., e_{2M}$ of $\h$. Then $T(\hat{\h}_{\Z+1/2})$ has a basis $u_1\otimes \cdots \otimes u_n$, where each $u_i$ is either of the form $e_{k_i}\otimes t^{m_i}$ or of the form $\mathbf{k}$. A pair of elements $(u_i, u_j)$ with $i<j$ is \textit{of wrong order}, if either of the following cases happen: (i) $u_i = e_{k_i}\otimes t^{m_i}, u_j = e_{k_j} \otimes t^{m_j}$ for some $k_i, k_j\in \{1, ..., 2M\}$ and some $m_i>0, m_j<0$; (ii) $u_i = \mathbf{k}$, $u_j =  e_{k_j}\otimes t^{m_j}$ for some $k_i\in \{1, ..., 2M\}, m_j \in \Z+1/2$. We define the \textit{defect} of $u_1\otimes \cdots \otimes u_n$ as the number of pairs of elements $(u_i, u_j)$ with $i<j$ that are of wrong order. 
    
    We define $L$ inductively on the basis elements $u_1 \otimes \cdots \otimes u_n$ by the following rule: $L$ acts on the basis elements with zero defect as identity. In particular, $L$ acts on the degree-zero tensors in $\C$ and the degree-1 tensors in $\hat{\h}_{\Z+1/2}$ as identity. For other basis elements, we define
    \begin{align*}
        & L(u_1\otimes \cdots \otimes (e_{k_i}\otimes t^{m_i+1/2}) \otimes (e_{k_{i+1}}\otimes t^{m_{i+1}+1/2}) \otimes  \cdots \otimes u_n) \\
        = \ &-L(u_1\otimes \cdots \otimes (e_{k_{i+1}}\otimes t^{m_{i+1}+1/2})\otimes (e_{k_i}\otimes t^{m_i+1/2}) \otimes \cdots \otimes u_n)\\
         & + m_{i}(e_{k_{i}}, e_{k_{i+1}}) \delta_{m_i+m_{i+1},0}L(u_1 \otimes \cdots \otimes u_{i-1} \otimes \mathbf{k} \otimes u_{i+2} \otimes \cdots \otimes u_n)
    \end{align*}
    whenever $m_{i}>0$ and $m_{i+1}<0$; we define
    \begin{align*}
         & L(u_1\otimes \cdots \otimes \mathbf{k} \otimes (e_{k_{i+1}}\otimes t^{m_{i+1}+1/2}) \otimes  \cdots \otimes u_n) \\
        = \ & L(u_1\otimes \cdots \otimes (e_{k_{i+1}}\otimes t^{m_{i+1}+1/2})\otimes \mathbf{k} \otimes \cdots \otimes u_n) 
    \end{align*}
    whenever $m_{i+1}\in \Z$. 
    
    We use induction to show that $L$ is well-defined. For basis elements of defect 0, the definition clearly has no ambiguities. For basis element with defect 1, if the pair $(u_i, u_j)$ is of wrong order, then necessarily, $j = i+1$. The definition would then have no ambiguities. Assume that $L$ has no ambiguity for all basis elements with smaller defects and for all tensors of lower degrees. Consider a basis element with defect at least 2. Necessarily, there exists two indices $i<j$ such that the pairs $(u_i, u_{i+1})$ and $(u_j, u_{j+1})$ are both of wrong order. 
    
    \begin{itemize}[leftmargin=*]
        \item If $i+1<j$, then we may express $L(u_1 \otimes \cdots \otimes u_i\otimes u_{i+1} \otimes \cdots \otimes u_j \otimes u_{j+1} \otimes \cdots \otimes u_n)$ as a unambiguous sum of $L(u_1\otimes \cdots \otimes u_{i+1} \otimes u_i \otimes \cdots \otimes u_{j+1}\otimes u_j \otimes \cdots \otimes u_n)$, which is of lower defect, and the images of $L$ on lower-degree tensors. From the induction hypothesis, they are all well-defined. For example, if $u_s = e_{k_s} \otimes t^{m_s}$ for $s = i, i+1, j, j+1$ with $m_{i}, m_{j} > 0, m_{i+1}, m_{j+1}< 0$, then 
    \begin{align*}
        & L(u_1\otimes \cdots \otimes u_i \otimes u_{i+1} \otimes \cdots \otimes u_{j}\otimes u_{j+1} \otimes \cdots \otimes u_n)\\
        = \ & - L(u_1\otimes \cdots \otimes u_{i+1} \otimes u_{i} \otimes \cdots \otimes u_{j}\otimes u_{j+1} \otimes \cdots \otimes u_n) \\ 
        & + m_i(e_{k_i}, e_{k_{i+1}})\delta_{m_i+m_{i+1}, 0} L(u_1\otimes \cdots \otimes u_{i-1} \otimes \mathbf{k} \otimes u_{i+2} \otimes \cdots \otimes u_{j}\otimes u_{j+1} \otimes \cdots \otimes u_n)\\
        = \ &  L(u_1\otimes \cdots \otimes u_{i+1} \otimes u_{i} \otimes \cdots \otimes u_{j+1}\otimes u_j \otimes \cdots \otimes u_n) \\ 
        & - m_j(e_{k_j}, e_{k_{j+1}})\delta_{m_j+m_{j+1}, 0} L(u_1\otimes \cdots \otimes u_{i-1} \otimes \mathbf{k} \otimes u_{i+2} \otimes \cdots \otimes u_{j-1} \otimes \mathbf{k}\otimes u_{j+2} \otimes \cdots \otimes u_n)\\
        & - m_i(e_{k_i}, e_{k_{i+1}})\delta_{m_i+m_{i+1}, 0} L(u_1\otimes \cdots \otimes u_{i-1} \otimes \mathbf{k} \otimes u_{i+2} \otimes \cdots \otimes u_{j}\otimes u_{j+1} \otimes \cdots \otimes u_n)\\
        & + m_i(e_{k_i}, e_{k_{i+1}})\delta_{m_i+m_{i+1}, 0}m_j(e_{k_j}, e_{k_{j+1}})\delta_{m_j+m_{j+1}, 0} \\
        & \qquad \cdot L(u_1\otimes \cdots \otimes u_{i-1} \otimes \mathbf{k} \otimes u_{i+2} \otimes \cdots \otimes u_{j-1} \otimes \mathbf{k}\otimes u_{j+2} \otimes \cdots \otimes u_n)
    \end{align*}
    Starting with $j$ instead of $i$ results in the same expression. So $L(u_1\otimes \cdots \otimes u_n)$ is unambiguously defined with the current choice of $u_s$. For other choices of $u_s$ that involve $\mathbf{k}$, we can proceed with a similar argument to get the same conclusion. Therefore, the inductive step is proved when $i+1 < j$. 
    \item If $i+1 = j$, then we may also express $L(u_1 \otimes \cdots \otimes u_i \otimes u_{i+1} \otimes u_{i+2} \otimes \cdots \otimes u_n)$ as an unambiguous sum of $L(u_1 \otimes \cdots u_{i+2} \otimes u_{i+1} \otimes u_i \otimes \cdots \otimes u_n)$, which is lower defect, and the images of $L$ on lower-degree tensors. For example, let $u_i = \mathbf{k}$, $u_t = e_{k_t}\otimes t^{m_{t}}$ for $t=i+1, i+2$ with $m_{i+1}>0, m_{i+2}<0$. If we first swap $u_i, u_{i+1}$, then swap $u_i, u_{i+2}$, finally swap $(u_{i+1}, u_{i+2})$, then we get
    \begin{align*}
        & L(u_1 \otimes \cdots \otimes u_i \otimes u_{i+1} \otimes u_{i+2} \otimes \cdots \otimes u_n) \\
        = \ & L(u_1 \otimes \cdots \otimes u_{i+1} \otimes u_{i} \otimes u_{i+2} \otimes \cdots \otimes u_n)\\
        = \ & L(u_1 \otimes \cdots \otimes u_{i+1} \otimes u_{i+2} \otimes u_{i} \otimes \cdots \otimes u_n)\\
        = \ & L(u_1 \otimes \cdots \otimes u_{i+2} \otimes u_{i+1} \otimes u_{i} \otimes \cdots \otimes u_n) \\
        & + m_{i+1}(e_{k_{i+1}}, e_{k_{i+2}}) \delta_{m_{i+1}+m_{i+2}, 0} L(u_1 \otimes \cdots \otimes u_{i-1} \otimes \mathbf{k} \otimes u_i \otimes  \cdots \otimes u_n)
    \end{align*}
    On the other hand, if we first swap $u_{i+1}, u_{i+2}$, then swap $u_i, u_{i+2}$, finally swap $u_i, u_{i+1}$, then we get
    \begin{align*}
        & L(u_1 \otimes \cdots \otimes u_i \otimes u_{i+1} \otimes u_{i+2} \otimes \cdots \otimes u_n) \\
        = \ & L(u_1 \otimes \cdots \otimes u_{i} \otimes u_{i+2} \otimes u_{i+1} \otimes \cdots \otimes u_n)\\
        & + m_{i+1}(e_{k_{i+1}}, e_{k_{i+2}}) \delta_{m_{i+1}+m_{i+2}, 0} L(u_1 \otimes \cdots \otimes u_{i-1} \otimes u_i \otimes \mathbf{k}  \otimes  \cdots \otimes u_n)\\
        = \ & L(u_1 \otimes \cdots \otimes u_{i+2} \otimes u_{i} \otimes u_{i+1} \otimes \cdots \otimes u_n)\\
        & + m_{i+1}(e_{k_{i+1}}, e_{k_{i+2}}) \delta_{m_{i+1}+m_{i+2}, 0} L(u_1 \otimes \cdots \otimes u_{i-1} \otimes \mathbf{k} \otimes u_i   \otimes  \cdots \otimes u_n)\\
        = \ & L(u_1 \otimes \cdots \otimes u_{i+2} \otimes u_{i+1} \otimes u_{i}  \otimes \cdots \otimes u_n)\\
        & + m_{i+1}(e_{k_{i+1}}, e_{k_{i+2}}) \delta_{m_{i+1}+m_{i+2}, 0} L(u_1 \otimes \cdots \otimes u_{i-1} \otimes \mathbf{k} \otimes u_i   \otimes  \cdots \otimes u_n)
    \end{align*}
    Starting with $j$ instead of $i$ results in the same expression. So $L(u_1\otimes \cdots \otimes u_n)$ is unambiguously defined with the current choice of $u_s$. For other possible choices of $u_s$, we can proceed with a similar argument to get the same conclusion. Therefore, the inductive step is proved when $i+1 = j$. 
    \end{itemize}
    Therefore, $L$ is a well-defined linear map. The conclusion then follows.    
\end{proof}

\subsection{The vector space $V$, the modes and the generating function of modes}

Let $\C \one$ be a one-dimensional vector space, on which $\mathbf{k}$ acts by the scalar $l\in \C$, $\hat{\h}_{\Z+1/2}^+$ act trivially. Then $\C \one$ is a module for the subalgebra $T(\hat{\h}_{\Z+1/2}^+) \otimes T(\C \mathbf{k})$ of $N(\hat{\h}_{\Z+1/2})$. 
Let 
$$V = N(\hat{\h}_{\Z+1/2})\otimes_{T(\hat{\h}_{\Z+1/2}^+) \otimes T(\C \mathbf{k})} \C\one$$
be the induced module. As a vector space, $V$ is isomorphic to $T(\hat{\h}_{\Z+1/2}^-)\otimes \C\one$. For $h\in \h, m\in \N$, we denote the action of $h\otimes t^{m+1/2}$ by $h(m+1/2)$. We refer $h(m+1/2)$ as positive modes if $m\geq 0$, negative modes if $m<0$. We use the brace brackets $\{\cdot, \cdot\}$ to denote the anti-commutator of the modes. Clearly, for $a, b\in \h, m,n\in \N$,
\begin{align}
    \{a(m+1/2), b(-n-1/2)\} &= a(m+1/2)b(-n-1/2) + b(-n-1/2)a(m+1/2) = (a, b) \delta_{mn} \label{anticom-modes}
\end{align}
Also, it follows from Proposition \ref{PBW} that $V$ is spanned by
\begin{align}
    h_1(-m_1-1/2) \cdots h_r(-m_r-1/2)\one, h_1, ..., h_r \in \h, m_1, ..., m_r\in \N, r\in \N. \label{V-basis}
\end{align}
We assign the half-integer
$$m_1 + \cdots + m_r + \frac r 2$$
as the weight of (\ref{V-basis}) and define the operator $\d$ accordingly. Clearly, for every $h\in \h$ and $n\in \Z$, the mode $h(n+1/2)$ is a homogeneous operator with weight $-n-1/2$.

Although we did not assume any relations among positive modes on $V$, it turns out that they are all anti-commutative. More precisely, 
\begin{prop}\label{pos-mode-prop}
    For $a,b\in \h$, $m,n\in \N$, 
    $$\{a(m+1/2), b(n+1/2)\} = a(m+1/2)b(n+1/2) + b(m+1/2)a(n+1/2) = 0$$
\end{prop}

\begin{proof}
    We check the action of the anti-commutator on the basis elements and show that 
    \begin{align}
        & a(m+1/2)b(n+1/2) h_1(-m_1-1/2)\cdots h_r(-m_r-1/2)\one \label{pos-mode-1}\\
        = \ &  -b(n+1/2)a(m+1/2) h_1(-m_1-1/2)\cdots h_r(-m_r-1/2)\one\label{pos-mode-2}
    \end{align}
    The identity clearly holds when $r= 0$. Assume the identity holds for all smaller $r$. We compute (\ref{pos-mode-1}) as follows
    \begin{align*}
        (\ref{pos-mode-1}) = \ & a(m+1/2) (b, h_1) \delta_{n, m_1} h_2(-m_2-1/2) \cdots h_r(-m_r-1/2)\one  \\
        & - a(m+1/2)h_1(-m_1-1/2) b(n+1/2)  h_2(-m_2-1/2) \cdots h_r(-m_r-1/2)\one\\
        = \ & a(m+1/2) (b, h_1) \delta_{n, m_1} h_2(-m_2-1/2) \cdots h_r(-m_r-1/2)\one  \\
        & - (a, h_1)\delta_{m,m_1} b(n+1/2)  h_2(-m_2-1/2) \cdots h_r(-m_r-1/2)\one\\
        & + h_1(-m_1-1/2) a(m+1/2)b(n+1/2)  h_2(-m_2-1/2) \cdots h_r(-m_r-1/2)\one\\
        (\ref{pos-mode-2})= \ & -b(n+1/2) (a,h_1)\delta_{m,m_1} h_2(-m_2-1/2)\cdots h_r(-m_r-1/2)\one \\
        & + b(n+1/2) h_1(-m_1-1/2) a(m+1/2) h_2(-m_2-1/2)\cdots h_r(-m_r-1/2)\one \\
        = \ & -b(n+1/2) (a,h_1)\delta_{m,m_1} h_2(-m_2-1/2)\cdots h_r(-m_r-1/2)\one \\
        & + (b, h_1)\delta_{n,m_1} a(m+1/2) h_2(-m_2-1/2)\cdots h_r(-m_r-1/2)\one \\
        & - h_1(-m_1-1/2) a(m+1/2) b(n+1/2) h_2(-m_2-1/2)\cdots h_r(-m_r-1/2)\one 
    \end{align*}
    The conclusion then follows from the induction hypothesis. 
\end{proof}

\begin{rema}
    We provide two conceptual explanations regarding the anti-commutativity of positive modes. Mathematically, this happens since $V$ is induced from a trivial module $\mathbb{C}\one$ for $T(\hat{\h}_{\Z+1/2}^+)$ where the anti-commutativity trivially holds. Physically, this means we may annihilating two interacting fermions regardless of the ordering. 
\end{rema}

For every $h\in \h$, we consider the series 
$$h(x) = \sum_{n\in \Z} h(n+1/2) x^{-n-1}\in \End(V)[[x, x^{-1}]]$$
that is the generating functions of all modes associated with $h$. Let 
\begin{align*}
    h(x)^+ &= \sum_{n\geq 0} h(n+1/2) x^{-n-1}\\
    h(x)^- &= \sum_{n\geq 0} h(-n-1/2) x^{n}
\end{align*}
be the partial series consisting of positive modes and negative modes, respectively. Contrary to the notation convention in \cite{LL}, $h(x)^+$ is the singular part of $h(x)$, while $h(x)^-$ is the regular part of $h(x)$. For every $n\in\Z$, we use the notation 
$$(x^{-n-1})^{(m)} = \frac{1}{m!}\frac{\partial^m}{\partial x^m} x^{-n-1} = \binom{-n-1}{m}x^{-n-m-1}. $$
We also use the notation
$$h^{(m)}(x) = \frac{1}{m!}\frac{\partial^m}{\partial x^m} h(x) = \sum_{n\in \Z} h(n+1/2)(x^{-n-1})^{(m)} = \sum_{n\in \Z} \binom{-n-1}{m} h(n+1/2)x^{-n-m-1}$$
Similarly, we have
\begin{align*}
    h^{(m)}(x)^+ &= \frac{1}{m!}\frac{\partial^m}{\partial x^m} h(x)^+ = \sum_{n\geq 0} h(n+1/2) (x^{-n-1})^{(m)}= \sum_{n\geq 0} \binom{-n-1}{m} h(n+1/2)x^{-n-m-1}\\
    h^{(m)}(x)^- &= \frac{1}{m!}\frac{\partial^m}{\partial x^m} h(x)^- =\sum_{n\geq 0} h(-n-1/2) (x^{n})^{(m)}= \sum_{n\geq 0} \binom{n}{m} h(n+1/2)x^{n-m}
\end{align*}

\begin{rema}
    Note that in Part II, we will use the same notation for the generating function of the modes acting on the $\Z_2$-twisted module, while use $h(x)_V$ to denote the generating function of the modes on $V$ we just defined above. We will trust the readers of the both parts not to confuse on these notations. 
\end{rema}

\begin{prop}
    Let 
    \begin{align*}
        f(x,y) &= (x-y)^{-1}, 
        f_{mn}(x,y) = \frac{1}{m!n!}\frac{\partial^{m+n}}{\partial x^m \partial y^n}f(x,y) = \binom{-n-1}{m}(x-y)^{-m-n-1}, m,n\geq 0. 
    \end{align*}
    Then for $a, b\in \h$, the anticommutator of $\{a^{(m)}(x)^+, b^{(n)}(y)^-\}$ is given by
    \begin{align}
        \{a^{(m)}(x)^+, b^{(n)}(y)^-\} &= (a,b) \iota_{xy}f_{mn}(x,y). \label{anticomm-pm}
    \end{align}
    Here $\iota_{xy}$ expands negative powers of $x-y$ as a power series in $y$, i.e., 
    $$\iota_{xy}\left(\frac{1}{(x-y)^t}\right) = \sum_{i=0}^\infty \binom{-t}i x^{t-i}(-y)^i = \sum_{i=0}^\infty \binom{t+i-1}{i} x^{t-i}y^i. $$
\end{prop}
\begin{proof}
    We first work out the case $m=n=0$. In this case, we compute the anti-commutator as follows:
    \begin{align*}
        \{a(x)^+, b(y)^-\} & = \sum_{i, j \geq 0} \{a(i+1/2), b(-j-1/2)\}x^{-i-1}y^j = \sum_{i,j \geq 0} (a, b)\delta_{ij}x^{-i-1}y^j\\
        & = \sum_{i\geq 0} (a, b)x^{-i-1}y^i = (a, b) \iota_{xy}(x-y)^{-1}. 
    \end{align*}
    The conclusion for general $m, n$ follows from taking partial derivatives. 
\end{proof}

\subsection{Normal ordering and vertex operator}
\begin{defn}\label{nord-defn-modes}
Let $h_1, ..., h_r\in \h$, $m_1, ..., m_r\in \Z$. We define the \textit{normal ordering} of the product $h_1(m_1+1/2), ..., h_r(m_r+1/2)$ of the modes by
\begin{align*}
    :h_1(m_1+1/2)\cdots h_r(m_r+1/2): = (-1)^\sigma h_{\sigma(1)}(m_{\sigma(1)}+1/2)\cdots  h_{\sigma(r)}(m_{\sigma(r)}+1/2) 
\end{align*}
where $\sigma$ is the unique permutation of $\{1, ..., r\}$ such that
\begin{align}
    m_{\sigma(1)}+1/2 < 0, ..., m_{\sigma(\eta)}+1/2 < 0, \ & m_{\sigma(\eta+1)}+1/2 > 0, ..., m_{\sigma(r)}+1/2 > 0 \\
    \sigma(1)< \cdots < \sigma(\eta), \ & \sigma(\eta+1)<\cdots < \sigma(r), \label{shuffle-defn}
\end{align}
$(-1)^\sigma$ is the parity of the permutation $\sigma$. 

Basically speaking, the normal ordered product of $h_1(m_1+1/2), ..., h_r(m_r+1/2)$ arranges all the negative modes to the left and all the positive modes to the right, meanwhile does not change the ordering among positive modes and that among negative modes. For example, if $m_1, m_2, m_3, m_4, m_5, m_6 \in \N$, then 
\begin{align*}
    & :h_1(-m_1-1/2)h_2(m_2+1/2)h_3(m_3+1/2)h_4(-m_4-1/2)h_5(m_5+1/2)h_6(-m_6-1/2):\\
    & \quad = - h_1(-m_1-1/2)h_4(-m_4-1/2)h_6(-m_6-1/2) h_2(m_2+1/2)h_3(m_3+1/2)h_5(m_5+1/2).
\end{align*}
\end{defn}

\begin{rema}
    Permutations satisfying (\ref{shuffle-defn}) is called a 2-shuffle. For every fixed $\eta\in \{0, ..., r\}$, the collection of permutations satisfying (\ref{shuffle-defn}) is denoted by $J_\eta(1, ..., r)$. We would also identify a 2-shuffle in $J_\eta(1, ... r)$ with two increasing sequences 
    $$p_1 < \cdots < p_\eta, p_1^c < \cdots < p_{r-\eta}^c$$
    in $\{1, ..., n\}$ that are complementary to each other, i.e., $\{p_1, ..., p_\eta\} \amalg \{p_1^c, ..., p_{r-\eta}^c\} = \{1, ..., r\}$. When $\eta =0$ or $\eta = r$, one of the increasing sequence is empty. In this case, $J_\eta(1, ..., r)$ consists of only the identity permutation $(1)$. We will not extensively use shuffles in Part I because there exists a recurrence relation of normal-ordered product that provides a more conceptual approach to the proof of the necessary identities (see Proposition \ref{recurrence-nord-prop}). We will have to use shuffles extensively in Part II when such recurrence does not exist. 
\end{rema}

\begin{defn}\label{nord-defn-series}
    The definition of the normal ordering extends naturally to all linear combinations of products of modes. With the same philosophy, we define the \textit{normal ordering} of the product series $h_1^{(m_1)}(x_1)\cdots  h_r^{(m_r)}(x_1)$ by components, i.e., 
    \begin{align*}
        & :h_1^{(m_1)}(x_1) \cdots h_r^{(m_r)}(x_r): \\
        = \ & \sum_{n_1, ..., n_r\in \Z} :h_1(n_1+1/2)\cdots h_r(n_r+1/2):(x_1^{(-n_1-1})^{(m_1)} \cdots (x_r^{(-n_r-1})^{(m_r)} 
    \end{align*}
\end{defn}

\begin{rema}
Clearly, the normal ordering operation the operation is ``multilinear'', in the sense that for $c_1, c_2\in \C$
\begin{align*}
    & :h_1^{(m_1)}(x_1) \cdots \left(c_1 h_{i_1}^{(m_{i_1})}(x_i) + c_2 h_{i_2}^{(m_{i_2})}(x_i)\right) \cdots  h_r^{(m_r)}(x_r):    \\
    = \ & c_1 :h_1^{(m_1)}(x_1) \cdots h_{i_1}^{(m_{i_1})}(x_i) \cdots  h_r^{(m_r)}(x_r): + c_2 :h_1^{(m_1)}(x_1) \cdots h_{i_2}^{(m_{i_2})}(x_i) \cdots  h_r^{(m_r)}(x_r): .
\end{align*}
    
\end{rema}

\begin{rema}\label{zero-suffices}
    Clearly, the normal ordering operation commutes with partial derivation: for every $i=1, ..., r$
$$\frac{\partial}{\partial x_i}:h_1^{(m_1)}(x_1) \cdots h_r^{(m_r)}(x_r): = :\frac{\partial}{\partial x_i}h_1^{(m_1)}(x_1) \cdots h_r^{(m_r)}(x_r):.  $$
This property simplifies the proofs of many identities in this paper. With this property, it suffices to show the identity in the case when $m_1 = \cdots = m_r = 0$. 

\end{rema}

\begin{defn}\label{vo-defn}
    We now define the vertex operator 
    $$Y: V\otimes V \to V((x))$$
    on the basis elements: For $\one \in V$, we set 
    $$Y(\one, x) = 1_V. $$
    For $h_1, ..., h_r\in \h, m_1, ..., m_r\in \N$, 
    \begin{align}
        & Y(h_1(-m_1-1/2)\cdots h_r(-m_r-1/2)\one, x) \nonumber\\
        = \ & :h_1^{(m_1)}(x)\cdots h_r^{(m_r)}(x):\nonumber\\
        = \ & \sum_{\eta = 0}^r \sum_{\sigma\in J_\eta(1, ..., r)} h_{\sigma(1)}^{(m_{\sigma(1)})}(x)^- \cdots h_{\sigma(\eta)}^{(m_{\sigma(\eta)})}(x)^- h_{\sigma(\eta+1)}^{(m_{\sigma(\eta+1)})}(x)^+ \cdots h_{\sigma(r)}^{(m_{\sigma(r)})}(x)^+ 
    \end{align}
    The vertex operator clearly extends to every element $u\in V$. It is also clear that for every fixed element $v\in V$, $Y(u, x)v\in V((x))$. 
\end{defn}

From Proposition \ref{pos-mode-prop}, we have the following recurrence relation of normal-ordered products, which will be useful in formulating the proofs. 
\begin{prop}\label{recurrence-nord-prop}
    For $r\in \Z_+$, $h_1, ..., h_r\in \h, m_1, ..., m_r \in \N$, 
    \begin{align}
        & :h_1^{(m_1)}(x_1)\cdots h_r^{(m_r)}(x_r): \nonumber\\ = \ & h_1^{(m_1)}(x_1)^-:h_2^{(m_2)}(x_2) \cdots h_r^{(m_r)}(x_r): \nonumber\\ 
        & + (-1)^{r-1} :h_2^{(m_2)}(x_2) \cdots h_r^{(m_r)}(x_r): h_1^{(m_1)}(x_1)^+. 
    \end{align}
\end{prop}

\begin{proof}
It suffices to prove the case when $m_1 = \cdots = m_r = 0$. The case with general nonzero $m_1, ..., m_r$ can be easily obtained by taking the partial derivatives. 
\begin{align}
    & :h_1(x_1) h_2(x_2) \cdots h_r(x_r):\nonumber \\
    = \ & :h_1(x_1)^- h_2(x_2) \cdots h_r(x_r): \label{recurrence-nord-1}\\
    & + :h_1(x_1)^+ h_2(x_2) \cdots h_r(x_r):\label{recurrence-nord-2}
\end{align}
We compute (\ref{recurrence-nord-1}) as follows
\begin{align*}
    (\ref{recurrence-nord-1})= \ & \sum_{n_1 \geq 0}\sum_{n_2, ..., n_r\in \Z} :h_1(-n_1-1/2) h_2(n_2+1/2)\cdots h_r(n_r+1/2):(x_1^{n_1})(x_2^{-n_2-1})\cdots (x_r^{-n_r-1})\\
    = \ & \sum_{n_1 \geq 0}\sum_{n_2, ..., n_r\in \Z} (x_1^{n_1})(x_2^{-n_2-1})\cdots (x_r^{-n_r-1})\\
    & \cdot \sum_{\eta=0}^r \sum_{\substack{\sigma\in J_{\eta}(1, ..., r)\\\sigma(1) = 1}}(-1)^\sigma h_1(-n_1-1/2) h_{\sigma(2)}(n_{\sigma(2)}+1/2)\cdots h_{\sigma(r)}(n_{\sigma(r)}+1/2)
\end{align*}
Since $\sigma(1)=1$, the lower limit of $\eta$ in the sum should be 1. We may regard $\sigma$ as a 2-shuffle in $J_{\eta-1}(2, ..., r)$ with the same parity. Summing over all $\sigma \in J(1, ..., r)$ with $\sigma(1)=1$ amounts to summing over all $\sigma\in J(2, ..., r)$. Then with an index shift of $\eta$, we see that  
\begin{align*}
    (\ref{recurrence-nord-1}) = \ & \sum_{n_1 \geq 0}\sum_{n_2, ..., n_r\in \Z} (x_1^{n_1})(x_2^{-n_2-1})\cdots (x_r^{-n_r-1})\\
    & \cdot \sum_{\eta=0}^{r-1} \sum_{\substack{\sigma\in J_{\eta}(2, ..., r)}}(-1)^\sigma h_1(-n_1-1/2) h_{\sigma(2)}(n_{\sigma(2)}+1/2)\cdots h_{\sigma(r)}(n_{\sigma(r)}+1/2)\\
    = \ & \sum_{n_1 \geq 0}h_1(-n_1-1/2)(x_1^{n_1}) \sum_{n_2, ..., n_r\in \Z} :h_2(n_2+1/2)\cdots h_r(n_r+1/2):(x_2^{-n_2-1})\cdots (x_r^{-n_r-1})\\
    = \ & h_1(x_1)^- : h_2(x_1) \cdots h_r(x_r):
\end{align*}
We compute (\ref{recurrence-nord-2}) as follows: 
\begin{align*}
    (\ref{recurrence-nord-2})= \ & \sum_{n_1 \geq 0}\sum_{n_2, ..., n_r\in \Z} :h_1(n_1+1/2) h_2(n_2+1/2)\cdots h_r(n_r+1/2):(x_1^{-n_1-1})(x_2^{-n_2-1})\cdots (x_r^{-n_r-1})\\
    = \ & \sum_{n_1 \geq 0}\sum_{n_2, ..., n_r\in \Z}  (x_1^{-n_1-1})(x_2^{-n_2-1})\cdots (x_r^{-n_r-1})\\
    & \cdot \sum_{\eta=0}^r \sum_{\substack{\sigma\in J_{\eta}(1, ..., r)\\\sigma(\eta+1)=1}}(-1)^\sigma h_{\sigma(1)}(n_{\sigma(1)}+1/2)\cdots h_{\sigma(\eta)}(n_{\sigma(\eta)}+1/2) h_1(n_1+1/2)\cdots   h_{\sigma(r)}(n_{\sigma(r)}+1/2)
\end{align*}
Using anti-commuativity of positive modes, we move $h_1(n_1+1/2)$ to the right, resulting in 
\begin{align*}
    & \sum_{n_1 \geq 0}\sum_{n_2, ..., n_r\in \Z}  (x_1^{n_1})(x_2^{-n_2-1})\cdots (x_r^{-n_r-1})\\
    & \cdot \sum_{\eta=0}^r \sum_{\substack{\sigma\in J_{\eta}(1, ..., r)\\\sigma(\eta+1)=1}}(-1)^\sigma h_{\sigma(1)}(n_{\sigma(1)}+1/2)\cdots h_{\sigma(\eta)}(n_{\sigma(\eta)}+1/2)\\
    & \qquad \qquad \qquad \quad \cdot (-1)^{r-\eta-1}h_{\sigma(\eta+2)}(n_{\sigma(\eta+2)}+1/2) \cdots   h_{\sigma(r)}(n_{\sigma(r)}+1/2)h_1(n_1+1/2)
\end{align*}
We now regard $\sigma(1), ..., \sigma(\eta), \sigma(\eta+2), ..., \sigma(r)$ as two increasing sequences in $\{2, ..., r\}$, which should correspond to a 2-shuffle $\tau\in J_{\eta}(2, ..., r)$ with 
$$\tau(2) = \sigma(1), ..., \tau(\eta+1) = \sigma(\eta), \tau(\eta+2) = \sigma(\eta+2), ..., \tau(r) = \sigma(r).$$
Regard $\tau$ as an element in $\Sym(1, ..., r)$, then we have 
$$\tau = \sigma \circ (\eta+1, \eta, ..., 2, 1).$$
As a result, we have 
$$(-1)^\tau = (-1)^\sigma \cdot (-1)^\eta. $$
Summing over all $\sigma\in J_\eta(1, ..., r)$ with $\sigma(\eta+1)=1$ amounts to summing over all $\tau\in J_\eta(2, ..., r)$. Clearly, the upper limit of $\eta$ is $r-1$ instead of $r$. Therefore, 
\begin{align*}
    (\ref{recurrence-nord-2}) = \ & \sum_{n_1 \geq 0}\sum_{n_2, ..., n_r\in \Z}  (x_1^{-n_1-1})(x_2^{-n_2-1})\cdots (x_r^{-n_r-1})\\
    & \cdot \sum_{\eta=0}^{r-1} \sum_{\substack{\tau\in J_{\eta}(2, ..., r)}}(-1)^\tau (-1)^\eta h_{\tau(2)}(n_{\tau(2)}+1/2)\cdots h_{\tau(\eta+1)}(n_{\tau(\eta+1)}+1/2)\\
    & \qquad \qquad \qquad \quad \cdot (-1)^{r-\eta-1}h_{\tau(\eta+2)}(n_{\tau(\eta+2)}+1/2) \cdots   h_{\tau(r)}(n_{\tau(r)}+1/2)h_1(n_1+1/2)\\
    = \ & (-1)^{r+1}\sum_{n_2, ..., n_r\in \Z} (x_2^{-n_2-1})\cdots (x_r^{-n_r-1})\\
    & \cdot \sum_{\eta=0}^{r-1} \sum_{\substack{\tau\in J_{\eta}(2, ..., r)}}(-1)^\tau  h_{\tau(2)}(n_{\tau(2)}+1/2)\cdots h_{\tau(r)}(n_{\tau(r)}+1/2)\sum_{n_1 \geq 0}h_1(n_1+1/2) (x_1^{-n_1-1})\\
    = \ & (-1)^{r-1} :h_2(x_2)\cdots h_r(x_r): h_1(x_1)^+. 
\end{align*}
\end{proof}

\begin{rema}
    The anti-commutativity of positive modes plays a crucial part in the proof of Proposition \ref{recurrence-nord-prop}. On the other hand, since we do not have anti-commutativity of negative modes, there does not exist a right-hand-recurrence that can express the normal-ordered product in terms of $:h_1^{(m_1)}(x_1)\cdots h_{r-1}^{(m_{r-1})}(x):$ and $h_r^{(m_r)}(x)$
\end{rema}

\begin{thm}\label{main}
    $(V, Y, \one)$ forms a $\frac \Z 2$-graded meromorphic open-string vertex algebra.
\end{thm}

\begin{proof}
    Here we check all axioms except the weak associativity, which will be addressed in the next section. 
    \begin{enumerate}[leftmargin=*]
        \item Axioms for the grading: Clearly the grading of $V$ is bounded below by 0. For the $\d$-commutator formula, we first compute the left-hand-side 
        \begin{align*}
            & [\d,  Y(h_1(-m_1-1/2)\cdots h_r(-m_r-1/2)\one, x)] \\
            = \ & \sum_{n_1, ..., n_r\in \Z} [\d, :h_1(n_1+1/2)\cdots h_r(n_r+1/2):] (x_1^{-n_1-1})^{(m_1)}\cdots (x_r^{-n_r-1})^{(m_r)}. 
        \end{align*}
        The weight of the endomorphism $h_1(n_1+1/2)\cdots h_r(n_r+1/2)$ is 
        $$\sum_{i=1}^r\left( -n_i - \frac 1 2\right).$$
        The normal ordering of the endomorphism does not change the weight. Therefore, 
        \begin{align*}
            & [\d,  Y(h_1(-m_1-1/2)\cdots h_r(-m_r-1/2)\one, x)] \\
            = \ &  \sum_{n_1, ..., n_r\in \Z} \sum_{i=1}^r\left(-n_i - \frac 1 2\right):h_1(n_1+1/2)\cdots h_r(n_r+1/2): (x_1^{-n_1-1})^{(m_1)}\cdots (x_r^{-n_r-1})^{(m_r)}. \\ 
            = \ & \sum_{i=1}^r\sum_{n_1, ..., n_r\in \Z} \left(-n_i - m_i - 1\right):h_1(n_1+1/2)\cdots h_r(n_r+1/2): (x_1^{-n_1-1})^{(m_1)}\cdots (x_r^{-n_r-1})^{(m_r)}. \\ 
            & + \sum_{i=1}^r\sum_{n_1, ..., n_r\in \Z} \left( m_i + \frac 1 2\right):h_1(n_1+1/2)\cdots h_r(n_r+1/2): (x_1^{-n_1-1})^{(m_1)}\cdots (x_r^{-n_r-1})^{(m_r)}\\
            = \ & x\frac d {dx} Y(h_1(-m_1-1/2)\cdots h_r(-m_r-1/2)\one x) +  Y(\d h_1(-m_1-1/2)\cdots h_r(-m_r-1/2)\one x) 
        \end{align*}
    \item Axioms for the vaccum: The identity property follows directly from the definition. For the creation property, notice that
    \begin{align*}
        & Y(h_1(-m_1-1/2)\cdots h_1(-m_r-1/2)\one, x) \one \\
        = \ & h_1^{(m_1)}(x)^-\cdots h_r^{(m_r)}(x)^- \one 
    \end{align*}
    contains only nonnegative powers of $x$. Note that $$\lim_{x\to 0}h_i^{(m_i)}(x)^- = \binom{m_i}{m_i} h_i(-m_i-1/2) = h_i(-m_i-1/2).$$ 
    Thus 
    \begin{align*}
        & \lim_{x\to 0}Y(h_1(-m_1-1/2)\cdots h_1(-m_r-1/2)\one, x) \one \\
        = \ & \lim_{x\to 0}h_1^{(m_1)}(x)^-\cdots h_r^{(m_r)}(x)^- \one\\
        = \ & h_1(-m_1-1/2)\cdots h_r(-m_r-1/2)\one. 
    \end{align*}
    \item $D$-derivative property and $D$-commutator formula: We first compute that 
    \begin{align*}
        & D h_1(-m_1-1/2)\cdots h_r(-m_r-1/2)\one \\
        = \ & \lim_{x\to 0}\frac d {dx} h_1^{(m_1)}(x)^-\cdots h_r^{(m_r)}(x)^- \one \\
        = \ & \lim_{x\to 0} \sum_{i=1}^r (m_i+1) h_1^{(m_1)}(x)^-\cdots h_i^{(m_i+1)}(x)^- \cdots h_r^{(m_r)}(x)^- \one\\
        = \ & \sum_{i=1}^r (m_i+1)h_1(-m_1-1/2)\cdots h_i(-m_i-3/2)\cdots h_r(-m_r-1/2)\one. 
    \end{align*}
    Thus
    \begin{align*}
        & Y(Dh_1(-m_1-1/2)\cdots h_r(-m_r-1/2)\one, x) \\
        = \ & \sum_{i=1}^r (m_i+1) Y(h_1(-m_1-1/2)\cdots h_i(-m_i-3/2) \cdots h_r(-m_r-1/2)\one, x)\\
        = \ & \sum_{i=1}^r (m_i+1) :h_1^{(m_1)}(x) \cdots h_i^{(m_i+1)}(x) \cdots h_r^{(m_r)}(x): \\
        = \ & \sum_{i=1}^r :h_1^{(m_1)}(x) \cdots \frac d {dx} h_i^{(m_i)}(x) \cdots h_r^{(m_r)}(x): \\
        = \ & \frac d {dx} :h_1^{(m_1)}(x) \cdots h_r^{(m_r)}(x): \\
        = \ & \frac d {dx} Y(h_1(-m_1-1/2)\cdots h_r(-m_r-1/2)\one, x)
    \end{align*}
    So the $D$-derivative property holds. 
    
    To check the $D$-commutator formula, we start by checking that for $h\in \h, m\in \Z$,     
    \begin{align}
        [D, h(m+1/2)] = -m h(m-1/2)\label{D-comm-1}
    \end{align}
    \begin{itemize}[leftmargin=*]
        \item In case $m+1/2>0$, we first compute the action of $D h(m+1/2)$ on a basis element: 
    \begin{align*}
        & D h(m+1/2) h_1(-m_1-1/2)\cdots h_r(-m_r-1/2) \one \\
        = \ & \sum_{i=1}^r (-1)^{i-1}\delta_{m,m_i} (h, h_i) \\
        & \qquad \cdot  D h_1(-m_1-1/2)\cdots \reallywidehat{h_i(-m_i-1/2)}\footnote{\textcolor{red}{Need to declare widehat notation somewhere}} \cdots h_r(-m_r-1/2) \one\\
        = \ & \sum_{i=1}^r \sum_{j=1}^{i-1} (m_j+1)(-1)^{i-1}\delta_{m,m_i}(h, h_i) \\
        & \qquad \qquad \cdot h_1(-m_1-1/2) \cdots h_j(m_j-3/2) \cdots \reallywidehat{h_i(-m_i-1/2)}\cdots h_r(-m_r-1/2) \one\\
        & + \sum_{i=1}^r \sum_{j=i+1}^{r} (m_j+1)(-1)^{i-1}\delta_{m,m_i}(h, h_i)  \\
        & \quad \qquad \qquad \cdot h_1(-m_1-1/2) \cdots \reallywidehat{h_i(-m_i-1/2)} \cdots h_j(m_j-3/2) \cdots h_r(-m_r-1/2) \one. 
    \end{align*}
    Now we compute the action of $h(m+1/2)D$ on a basis element: 
    \begin{align*}
        & h(m+1/2) D h_1(-m_1-1/2)\cdots h_r(-m_r-1/2)\one \\
        = \ & \sum_{j=1}^r (m_j+1) h(m+1/2)h_1(-m_1-1/2)\cdots h_j(-m_j-3/2) \cdots h_r(-m_r-1/2)\one \\
        = \ & \sum_{j=1}^r \sum_{i=1}^{j-1} (m_j+1) (-1)^{i-1}\delta_{m,m_i} (h, h_i)\\
        & \qquad \quad \cdot h_1(-m_1-1/2) \cdots \reallywidehat{h_i(m_i-1/2)} \cdots h_j(-m_j-3/2) \cdots h_r(-m_r-1/2)\one \\
        & + \sum_{j=1}^r (m_j+1) \delta_{m, m_j+1} (-1)^{j-1}(h, h_j) \\
        & \qquad \quad \cdot h_1(-m_1-1/2)\cdots \reallywidehat{h_j(-m_j-1/2)} \cdots h_r(-m_r-1/2)\one \\
        & + \sum_{j=1}^r \sum_{i=j+1}^r 
        (m_j+1) (-1)^{i-1}\delta_{m, m_i}(h, h_i) \\
        & \qquad \qquad \quad \cdot h_1(-m_1-1/2) \cdots h_j(-m_j-3/2) \cdots \reallywidehat{h_i(-m_i-1/2)} \cdots h_r(-m_r-1/2)\one. 
    \end{align*}
    Since 
    $$\sum_{i=1}^r \sum_{j=1}^{i-1} = \sum_{1\leq j < i \leq r} = \sum_{j=1}^r \sum_{i=j+1}^r, \sum_{i=1}^r \sum_{j=i+1}^{r} = \sum_{1\leq i < j \leq r} = \sum_{j=1}^r \sum_{i=1}^{j-1}$$
    Thus 
    \begin{align*}
        & [D, h(m+1/2)]h_1(-m_1-1/2)\cdots h_r(-m_r-1/2)\one \\
        = \ & -\sum_{j=1}^r (m_j+1)\delta_{m,m_j+1}(-1)^{j-1} (h, h_j)\\
        & \qquad \quad \cdot h_1(-m_1-1/2) \cdots \reallywidehat{h_j(-m_j-1/2)} \cdots h_r(-m_r-1/2)\\
        = \ & -\sum_{j=1}^r (m-1+1)\delta_{m-1,m_j}(-1)^{j-1} (h, h_j) \\
        & \qquad \quad \cdot h_1(-m_1-1/2) \cdots \reallywidehat{h_j(-m_j-1/2)} \cdots h_r(-m_r-1/2)\\
        = \ & -mh(m-1/2) h_1(-m_1-1/2)\cdots h_r(-m_r-1/2)
    \end{align*}
    So we checked (\ref{D-comm-1}) when $m+1/2 > 0$. 
        \item In case $m+1/2<0$, we rewrite $h(m+1/2)$ as $h(-n-1/2)$ with $n = -m -1$ and compute that
    \begin{align*}
        & D h(-n-1/2)h_1(-m_1-1/2)\cdots h_r(-m_r-1/2) \one\\
        = \ &  (n+1)h(-n-3/2)  h_1(-m_1-1/2)\cdots h_r(-m_r-1/2) \one \\
        & + h(-n-1/2)\sum_{i=1}^r (m_i+1/2) h_1(-m_1-1/2)\cdots h_i(-m_i-3/2) \cdots h_r(-m_r-1/2)\one \\
        = \ &  (n+1)h(-n-3/2)  h_1(-m_1-1/2)\cdots h_r(-m_r-1/2) \one \\
        & + h(-n-1/2)D h_1(-m_1-1/2)\cdots h_r(-m_r-1/2) \one
    \end{align*}
    Therefore, 
    \begin{align*}
        & [D, h(-n-1/2)]h_1(-m_1-1/2)\cdots h_r(-m_r-1/2)\one \\
        = \ & (n+1)h(-n-3/2)h_1(-m_1-1/2)\cdots h_r(-m_r-1/2)\one. 
    \end{align*}
    Substituting $n=-m-1$, we obtain (\ref{D-comm-1}) when $m+1/2<0$. 
    \end{itemize}

    \vskip 1em

    From (\ref{D-comm-1}), we immediately see that
    \begin{align*}
        [D, h^{(m)}(x)] = \ & \sum_{n\in \Z} [D, h(n+1/2)] \binom{-n-1} m x^{-n-m-1} =  \sum_{n\in \Z}-n h(n-1/2) \binom{-n-1} m x^{-n-m-1}\\
        = \ & \sum_{n\in \Z}(-n-1) h(n+1/2) \binom{-n-2} m x^{-n-m-2} =  \sum_{n\in \Z}(-n-m-1) h(n+1/2) \binom{-n-1} m x^{-n-m-2}\\
        = \ & \frac{d}{dx} \sum_{n\in \Z} h(n+1/2) \binom {-n-1}{m} x^{-n-1-m}   = \frac d {dx} h^{(m)}(x)
    \end{align*}
    So the $D$-commutator formula holds for $Y(h(-m-1/2)\one, x)$ with one modes. Similarly, we have
    $$[D, h^{(m)}(x)^+] = \frac{d}{dx} h^{(m)}(x)^+, [D, h^{(m)}(x)^-] = \frac{d}{dx} h^{(m)}(x)^-. $$
    
    Now we use induction to prove the $D$-commutator formula with an arbitrary number of modes. 
    \begin{align*}
        & [D, Y(h_1(-m_1-1/2)\cdots h_r(-m_r-1/2)\one, x)] \\
        = \ & [D, :h_1^{(m_1)}(x) \cdots h_r^{(m_r)}(x):]\\
        = \ & [D, h_1^{(m_1)}(x)^-:h_2^{(m_2)}(x) \cdots h_r^{(m_r)}(x):] + [D, :h_2^{(m_2)}(x) \cdots h_r^{(m_r)}(x): h_1^{(m_1)}(x)^+]\\
        = \ & [D, h_1^{(m_1)}(x)^-] :h_2^{(m_2)}(x) \cdots h_r^{(m_r)}(x): + h_1^{(m_1)}(x) [D, :h_2^{(m_2)}(x)\cdots h_r^{m_r}(x):]\\
        & + [D, :h_2^{(m_2)}(x) \cdots h_r^{(m_r)}(x):] h_1^{(m_1)}(x)^+ + :h_2^{(m_2)}(x) \cdots h_r^{(m_r)}(x):[D,  h_1^{(m_1)}(x)^+]\\
        = \ & \frac{d}{dx} \left(h_1^{(m_1)}(x)^-\right):h_2^{(m_2)}(x)\cdots h_r^{m_r}(x): + h_1^{(m_1)}(x)^-\frac{d}{dx}\left(:h_2^{(m_2)}(x)\cdots h_r^{(m_r)}(x):\right)\\
        & + \frac d{dx}\left(:h_2^{(m_2)}(x) \cdots h_r^{(m_r)}(x):\right) h_1^{(m_1)}(x)^+ +  :h_2^{(m_2)}(x) \cdots h_r^{(m_r)}(x):\frac{d}{dx}\left(h_1^{(m_1)}(x)^+\right)\\
        = \ & \frac{d}{dx} \left(h_1^{(m_1)}(x)^-:h_2^{(m_2)}(x)\cdots h_r^{m_r}(x):+:h_2^{(m_2)}(x) \cdots h_r^{(m_r)}(x): h_1^{(m_1)}(x)^+ \right)\\
        = \ & \frac{d}{dx}:h_1^{(m_1)}(x)\cdots h_r^{(m_r)}(x):. 
    \end{align*}
    \end{enumerate}
        
\end{proof}

\section{Proof of weak associativity}

In this section we prove the weak associativity by explicitly computing the product and iterates of two vertex operators. 

\subsection{Product of two normal-ordered products}

We start by the following special case. 

\begin{prop}\label{Product-1-s-prop}
    Let $a, b_1, ..., b_s \in \h, m, n_1, ..., n_s\in \N$, then 
    \begin{align}
        & a^{(m)}(x) : b_1^{(n_1)}(y_1)\cdots b_s^{(n_s)}(y_s): \nonumber\\
        = \ & :a^{(m)}(x)b_1^{(n_1)}(y_1)\cdots b_s^{(n_s)}(y_s): \nonumber\\
        & + \sum_{i=1}^{s} (-1)^{i-1} (a, b_i) \iota_{xy}f_{mn_i}(x,y) :b_1^{(n_1)}(y_1) \cdots \reallywidehat{b_i^{(n_i)}(y_i)} \cdots b_s^{(n_s)}(y_s): \label{Product-1-s-formula}
    \end{align}
\end{prop}

\begin{proof}
        From Remark \ref{zero-suffices}, it suffices to show the case for $m = n_1 = \cdots = n_s = 0$. We prove by induction. In case $s=1$, we have
    \begin{align*}
        & a(x) :b_1(y_1): = a(x) b_1(y_1) \\
        = \ & a(x)^- b_1(y_1) + a(x)^+ b_1(y_1)^- + a(x)^+ b_1(y_1)^+\\
        = \ & a(x)^- b_1(y_1) - b_1(y_1)^- a(x)^+ + \{a(x)^+, b_1(y_1)^-\} - b_1(y_1)^+ a(x)^+ \\
        = \ & a(x)^- b_1(y_1) - b_1(y_1) a(x)^+ + (a,b)\iota_{xy_1}f(x,y_1)\\
        = \ & :a(x)b_1(y_1): + (a,b)\iota_{xy_1}f(x,y_1).
    \end{align*}
    So the conclusion holds when $s=1$. Assume the conclusion holds for all smaller $s$. We use Proposition \ref{recurrence-nord-prop} to expand the normal-ordered product of $b_1(y_1), ..., b_s(y_s)$. So
    \begin{align}
        a(x) :b_1(y_1)\cdots b_s(y_s):  = \ & a(x) b_1(y_1)^-:b_2(y_2)\cdots b_s(y_s): + (-1)^{s-1}a(x) :b_2(y_2)\cdots b_s(y_s): b_1(y_1)^+ \nonumber\\
        = \ & a(x)^- b_1(y_1)^-:b_2(y_2)\cdots b_s(y_s): \label{Product-1-s-1}\\
        & + a(x)^+b_1(y_1)^- :b_2(y_2)\cdots b_s(y_s): \label{Product-1-s-2}\\
        & + (-1)^{s-1}a(x) :b_2(y_2)\cdots b_s(y_s): b_1(y_1)^+\label{Product-1-s-3}
    \end{align}
    We first focus on (\ref{Product-1-s-2}). 
    \begin{align}
        (\ref{Product-1-s-2}) = \ &  - b_1(y_1)^- a(x)^+:b_2(y_2)\cdots b_s(y_s): + (a,b_1)\iota_{xy_1}f(x,y_1):b_2(y_2)\cdots b_s(y_s):\nonumber\\
        = \ & -b_1(y_1)^- a(x):b_2(y_2)\cdots b_s(y_s):\label{Product-1-s-2-0}\\
        & + b_1(y_1)^- a(x)^-:b_2(y_2)\cdots b_s(y_s): + (a,b_1)\iota_{xy_1}f(x,y_1):b_2(y_2)\cdots b_s(y_s):  \nonumber
    \end{align}
    Apply induction to handle (\ref{Product-1-s-2-0}), 
    \begin{align}
        (\ref{Product-1-s-2}) = \ & -b_1(y_1)^- :a(x) b_2(y_2)\cdots b_s(y_s):\label{Product-1-s-2-1}\\
        & -b_1(y_1)^- \sum_{j=2}^s (-1)^j (a,b_j)\iota_{xy_j}f(x,y_j) :b_2(y_2)\cdots \reallywidehat{b_j(y_j)} \cdots b_s(y_s): \label{Product-1-s-2-2}\\
        &  +b_1(y_1)^- a(x)^-:b_2(y_2)\cdots b_s(y_s):\label{Product-1-s-2-3}\\
        & + (a,b_1)\iota_{xy_1}f(x,y_1):b_2(y_2)\cdots b_s(y_s):  \label{Product-1-s-2-4}
    \end{align}
    We now focus on (\ref{Product-1-s-3}).   
    Apply the induction hypothesis, we have
    \begin{align}
        (\ref{Product-1-s-3}) = \ & (-1)^{s-1}:a(x) b_2(y_2)\cdots b_s(y_s): b_1(y_1)^+ \label{Product-1-s-3-1}\\
        & + (-1)^{s-1} \sum_{j=2}^s (-1)^{j} (a,b_j)\iota_{xy_j}f(x,y_j):b_2^{n_2}(y_2)\cdots \reallywidehat{b_j(y_j)}\cdots b_s(y_s):b_1(y_1)^+ \label{Product-1-s-3-2}
    \end{align}
    Therefore, $(\ref{Product-1-s-1})+(\ref{Product-1-s-2})+(\ref{Product-1-s-3}) = (\ref{Product-1-s-1})+(\ref{Product-1-s-2-1})+(\ref{Product-1-s-2-2})+(\ref{Product-1-s-2-3})+(\ref{Product-1-s-2-4}) + (\ref{Product-1-s-3-1})+ (\ref{Product-1-s-3-2})$. Note that 
    \begin{align}
        & (\ref{Product-1-s-2-4})+ (\ref{Product-1-s-2-2}) + (\ref{Product-1-s-3-2}) \nonumber\\
        = \ & (a, b_1) \iota_{xy_1}f(x,y_1):b_2(y_2) \cdots b_s(y_s): \nonumber 
        \\
        & -\sum_{j-2}(-1)^j (a,b_j)\iota_{xy_j}f(x,y_j) \nonumber\\
        & \quad \cdot \left( b_1(y_1)^-:b_2(y_2)\cdots \reallywidehat{b_j(y_j)}b_s(y_s): + (-1)^{s-2} b_2(y_2)\cdots \reallywidehat{b_j(y_j)}b_s(y_s): b_1(y_1)^+\right)\nonumber \\
        = \ & (a, b_1) \iota_{xy_1}f(x,y_1):b_2(y_2) \cdots b_s(y_s):\nonumber\\
        & + \sum_{j-2}(-1)^{j-1} (a,b_j)\iota_{xy_j}f(x,y_j) :b_1(y_1)\cdots \reallywidehat{b_j(y_j)}\cdots b_s(y_s): \nonumber\\
        = \ & \sum_{j-1}(-1)^{j-1} (a,b_j)\iota_{xy_j}f(x,y_j) :b_1(y_1)\cdots \reallywidehat{b_j(y_j)}\cdots b_s(y_s):.\label{Product-1-s-4}
    \end{align}
    Note also that 
    \begin{align}
        & (\ref{Product-1-s-1}) + (\ref{Product-1-s-2-1})+(\ref{Product-1-s-2-3})+(\ref{Product-1-s-3-1})\nonumber \\
        = \ & a(x)^- b_1(y_1)^-:b_2(y_2)\cdots b_s(y_s):  -b_1(y_1)^- :a(x) b_2(y_2)\cdots b_s(y_s):\nonumber\\
        & +b_1(y_1)^- a(x)^-:b_2(y_2)\cdots b_s(y_s): +(-1)^{s-1}:a(x) b_2(y_2)\cdots b_s(y_s): b_1(y_1)^+\nonumber\\
        = \ & a(x)^- b_1(y_1)^-:b_2(y_2)\cdots b_s(y_s): \nonumber\\
        & -b_1(y_1)^- \left(a(x)^-: b_2(y_2)\cdots b_s(y_s):-(-1)^{s-1}: b_2(y_2)\cdots b_s(y_s):a(x)^+\right)\nonumber\\
        & +b_1(y_1)^- a(x)^-:b_2(y_2)\cdots b_s(y_s):\nonumber\\
        & +(-1)^{s-1}\left(a(x)^- :b_2(y_2)\cdots b_s(y_s): b_1(y_1)^+ + (-1)^{s-1} :b_2(y_2)\cdots b_s(y_s): a(x)^+b_1(y_1)^+\right)\nonumber\\
        = \ & a(x)^- b_1(y_1)^-:b_2(y_2)\cdots b_s(y_s):  -(-1)^{s-1} b_1(y_1)^-:b_2(y_2)\cdots b_s(y_s):a(x)^+\nonumber\\
        & +(-1)^{s-1}a(x)^- :b_2(y_2)\cdots b_s(y_s): b_1(y_1)^+ - :b_2(y_2)\cdots b_s(y_s): b_1(y_1)^+a(x)^+\nonumber \\
        = \ & a(x)^- :b_1(y_1)b_2(y_2)\cdots b_s(y_s): -(-1)^{s-1} :b_1(y_1)b_2(y_2)\cdots b_s(y_s):a(x)^+\nonumber\\
        = \ & :a(x)b_1(y_1)b_2(y_2)\cdots b_s(y_s): \label{Product-1-s-5}
    \end{align}
    Thus we managed to show that the left-hand-side of (\ref{Product-1-s-formula}) is precisely $(\ref{Product-1-s-4}) + (\ref{Product-1-s-5})$, which is precisely the right-hand-side of (\ref{Product-1-s-formula}). 
\end{proof}

\begin{prop}\label{Product-r-1-prop}
    Let $a_1, ..., a_r, b\in \h, m_1, ..., m_r, n\in \N$, then 
    \begin{align}
        & :a_1^{(m_1)}(x_1)\cdots a_r^{(m_r)}(x_r) : b^{(n)}(y) \nonumber\\
        = \ & :a^{(m_1)}(x_1)\cdots a_r^{(m_r)}(x_r)b^{(n)}(y): \nonumber\\
        & + \sum_{i=1}^{r} (-1)^{i-r} (a_i, b) \iota_{xy}f_{m_in}(x,y) :a_1^{(m_1)}(x_1) \cdots \reallywidehat{a_i^{(m_i)}(x_i)} \cdots a_r^{(n_r)}(x_r): \label{Product-r-1-formula}
    \end{align}
\end{prop}

\begin{proof}
    The proof is significantly simpler than that for Proposition \ref{Product-r-1-prop}.  From Remark \ref{zero-suffices}, it suffices to show the case for $m_1 = \cdots = m_r = n= 0$. We prove by induction. The case $r=1$ is clear from Proposition \ref{Product-r-1-prop}. 
    Assume the conclusion holds for all smaller $r$. We use Proposition \ref{recurrence-nord-prop} to expand the normal ordering of $a_1(x_1)\cdots a_r^{(n_r)}(x_r)$. So
    \begin{align}
        & :a_1(x_1)\cdots a_r(x_r): b^{(n)}(y) \nonumber\\
        = \ & a_1(x_1)^-:a_2(x_2)\cdots a_r(x_r): b(y) + (-1)^{r-1} :a_2(x_2)\cdots a_r(x_r): a_1(x_1)^+ b(y)\nonumber\\
        = \ & a_1(x_1)^-:a_2(x_2)\cdots a_r(x_r): b(y) \nonumber\\
        & + (-1)^{r-1} :a_2(x_2)\cdots a_r(x_r): a_1(x_1)^+ b(y)^+ + (-1)^{r-1} :a_2(x_2)\cdots a_r(x_r): a_1(x_1)^+ b(y)^-  \nonumber\\
        = \ & a_1(x_1)^-:a_2(x_2)\cdots a_r(x_r): b(y) \nonumber\\
        & + (-1)^{r} :a_2(x_2)\cdots a_r(x_r): b(y)^+a_1(x_1)^+  + (-1)^{r} :a_2(x_2)\cdots a_r(x_r): b(y)^-a_1(x_1)^+   \nonumber\\
        & + (-1)^{r-1}:a_2(x_2)\cdots a_r(x_r): (a_1, b)\iota_{x_1y}f(x_1,y)\nonumber \\
        = \ & a_1(x_1)^-:a_2(x_2)\cdots a_r(x_r): b(y) \label{Product-r-1-1} + (-1)^{r} :a_2(x_2)\cdots a_r(x_r): b(y)a_1(x_1)^+  \\
        & + (-1)^{r-1}(a_1, b)\iota_{x_1y}f(x_1,y) :a_2(x_2)\cdots a_r(x_r):\label{Product-r-1-2}
    \end{align}
    Apply induction on (\ref{Product-r-1-1}) to see that
    \begin{align}
        (\ref{Product-r-1-1}) = \ & a_1(x_1)^- :a_2(x_2)\cdots a_r(x_r) b(y): \nonumber \\
        & + \sum_{i=2}^r (-1)^{i-1-(r-1)}(a_i, b)\iota_{x_iy}f(x_i, y) a_1(x_1)^- :a_2(x_2) \cdots \reallywidehat{a_i(x_i)}\cdots a_r(x_r): \nonumber \\
        & + (-1)^r :a_2(x_2)\cdots a_r(x_r) b(y): a_1(x_1)^+  \nonumber \\
        & + (-1)^r \sum_{i=2}^r (-1)^{i-1-(r-1)}(a_i, b)\iota_{x_iy}f(x_i, y)  :a_2(x_2) \cdots \reallywidehat{a_i(x_i)}\cdots a_r(x_r): a_1(x_1)^+\nonumber \\
        = \ & :a_1(x_1)a_2(x_2)\cdots a_r(x_r) b(y): \label{Product-r-1-3}\\
        & +  \sum_{i=2}^r (-1)^{i-r}(a_i, b)\iota_{x_iy}f(x_i, y)  :a_1(x_1)a_2(x_2) \cdots \reallywidehat{a_i(x_i)}\cdots a_r(x_r): \label{Product-r-1-4}
    \end{align}
    So the left-hand-side of (\ref{Product-r-1-formula}) is precisely $(\ref{Product-r-1-1}) + (\ref{Product-r-1-3}) + (\ref{Product-r-1-4})$, which is precisely the right-hand-side of (\ref{Product-r-1-formula}). 
\end{proof}

Now we study the general case
\begin{thm}\label{Product-r-s-thm}
    Let $a_1, ..., a_r, b_1, ..., b_s\in \h, m_1, ..., m_r, n_1, ..., n_s\in \N$. Then 
    \begin{align}
        & :a_1^{(m_1)}(x_1)\cdots a_r^{(m_r)}(x_r): \cdot :b_1^{(n_1)}(y_1)\cdots b_s^{(n_s)}(y_s): \nonumber \\
        = \ & :a_1^{(m_1)}(x_1)\cdots a_r^{(m_r)}(x_r)b_1^{(n_1)}(y_1)\cdots b_s^{(n_s)}(y_s):\nonumber \\
        & + \sum_{\rho=1}^{\min(r,s)}\sum_{\substack{ 1 \leq i_1 < \cdots < i_\rho\leq r\\ 1 \leq j_1 < \cdots < j_\rho \leq s}}(-1)^{i_1+\cdots + i_\rho + j_1+\cdots + j_\rho }(-1)^{r\rho + \rho(\rho+1)/2}\nonumber \\
        & \hspace{9.7 em}\cdot \begin{vmatrix}
            (a_{i_1}, b_{j_1})f_{m_{i_1}n_{j_1}}(x_{i_1}, y_{j_1}) & \cdots & (a_{i_1}, b_{j_\rho})f_{m_{i_1}n_{j_\rho}}(x_{i_1}, y_{j_\rho})\\
            \vdots & & \vdots\nonumber  \\
            (a_{i_\rho}, b_{j_1})f_{m_{i_\rho}n_{j_1}}(x_{i_\rho}, y_{j_1}) & \cdots & (a_{i_\rho}, b_{j_\rho})f_{m_{i_\rho}n_{j_\rho}}(x_{i_\rho}, y_{j_\rho})
        \end{vmatrix}\nonumber \\
        & \hspace{10 em}\cdot :a_{1}^{(m_1)}(x_1) \cdots \reallywidehat{a_{i_1}^{(m_{i_1})}(x_{i_1})} \cdots \reallywidehat{a_{i_\rho}^{(m_{i_\rho})}(x_{i_\rho})} \cdots a_{r}^{(m_r)}(x_r)\nonumber  \\
        & \hspace{11 em}\cdot b_{1}^{(n_1)}(y_1) \cdots \reallywidehat{b_{j_1}^{(n_{j_1})}(y_{j_1})} \cdots \reallywidehat{b_{j_\rho}^{(n_{j_\rho})}(y_{j_\rho})} \cdots b_{s}^{(n_s)}(y_s):\label{Product-r-s-formula}     
    \end{align}
\end{thm}

\begin{proof}
From Remark \ref{zero-suffices}, it suffices to argue the case when $m_1 = \cdots = m_r = n_1 = \cdots = n_s = 0$. The conclusion clearly holds for $r=1$ and arbitrary $s$. Assume the conclusion holds for $r$ and all smaller $r$ and arbitrary $s$, we argue the $r+1$ case. First assume $r+1\leq s$
\begin{align}
    & : a_1(x_1) \cdots a_r(x_r) a_{r+1}(x_{r+1}) :\cdot : b_1(y_1) \cdots b_s(y_s) : \nonumber\\
    = & : a_1(x_1) \cdots a_r(x_r): a_{r+1}(x_{r+1})  : b_1(y_1) \cdots b_s(y_s) :\nonumber\\
    & - \sum_{i=1}^r (-1)^{r-i} (a_i, a_{r+1}) f(x_i, x_{r+1}) : a_1(x_1) \cdots \reallywidehat{a_i(x_i)} \cdots a_r(x_r): : b_1(y_1) \cdots b_s(y_s):\nonumber\\
     = & : a_1(x_1) \cdots a_r(x_r): : a_{r+1}(x_{r+1})   b_1(y_1) \cdots b_s(y_s) : \label{Prod-Line-1}\\
    & + \sum_{j=1}^s (-1)^{j-1} (a_{r+1}, b_j)f(x_{r+1}, y_j) : a_1(x_1)\cdots a_r(x_r): : b_1(y_1) \cdots \reallywidehat{b_j(y_j)} \cdots b_s(y_s):  \label{Prod-Line-2}\\ 
    & - \sum_{i=1}^r (-1)^{r-i} (a_i, a_{r+1}) f(x_i, x_{r+1}) : a_1(x_1) \cdots \reallywidehat{a_i(x_i)} \cdots a_r(x_r): : b_1(y_1) \cdots b_s(y_s) : \label{Prod-Line-3}
\end{align}
By induction hypothesis, (\ref{Prod-Line-1}) is equal to 
\begin{align}
    : a_1(x_1) \cdots a_{r}(x_r)a_{r+1}(x_{r+1})\cdots b_1(y_1)\cdots b_s(y_s): \label{Prod-Fully-Normal}
\end{align}together with the following correction terms: 
\begin{align}
    \sum_{\rho=1}^{r} \sum_{1\leq i_1 < \cdots < i_\rho \leq r} \sum_{1\leq j_1 < \cdots < j_\rho \leq s} & (-1)^{i_1+\cdots +i_\rho + j_1 + 1+ \cdots + j_\rho + 1}(-1)^{r\rho+\frac{\rho(\rho+1)}2}\nonumber\\
    \cdot & \begin{vmatrix}
    (a_{i_1}, b_{j_1})f(x_{i_1}, y_{j_1}) & \cdots & (a_{i_1}, b_{j_\rho})f(x_{i_1}, y_{j_\rho}) \nonumber\\
    \vdots & & \vdots \nonumber\\
    (a_{i_\rho}, b_{j_1})f(x_{i_\rho}, y_{j_1}) & \cdots & (a_{i_1}, b_{j_\rho})f(x_{i_1}, y_{j_\rho}) 
    \end{vmatrix}\nonumber\\
    \cdot & : a_1(x_1) \cdots \reallywidehat{a_{i_1}(x_{i_1})} \cdots  \reallywidehat{a_{i_\rho}(x_{i_\rho})} \cdots a_r(x_r) a_{r+1}(x_{r+1})\nonumber\\
    & \quad  \cdot b_1(y_1) \cdots \reallywidehat{b_{j_1}(y_{j_1})} \cdots \reallywidehat{b_{j_\rho}(y_{j_\rho})} \cdots b_s(y_s) : \nonumber\\
     \quad +\sum_{\rho=1}^{r} \sum_{1\leq i_1 < \cdots < i_\rho \leq r} \sum_{1\leq j_2 < \cdots < j_\rho \leq s} & (-1)^{i_1+\cdots +i_\rho + 1 + j_2 + 1+ \cdots + j_\rho + 1}(-1)^{r\rho+\frac{\rho(\rho+1)}2}\nonumber\\
    \cdot & \begin{vmatrix}
    (a_{i_1}, a_{r+1})f(x_{i_1}, x_{r+1}) & (a_{i_1}, b_{j_2})f(x_{i_1}, y_{j_2}) & \cdots & (a_{i_1}, b_{j_\rho})f(x_{i_1}, y_{j_\rho}) \\
    \vdots & \vdots & & \vdots \\
    (a_{i_\rho}, a_{r+1})f(x_{i_\rho}, x_{r+1})& (a_{i_\rho}, b_{j_2})f(x_{i_\rho}, y_{j_2}) & \cdots & (a_{i_1}, b_{j_\rho})f(x_{i_1}, y_{j_\rho}) 
    \end{vmatrix}\nonumber\\
    \cdot & : a_1(x_1) \cdots \reallywidehat{a_{i_1}(x_{i_1})} \cdots  \reallywidehat{a_{i_\rho}(x_{i_\rho})} \cdots a_r(x_r) \nonumber\\
    & \quad \cdot b_1(y_1) \cdots \reallywidehat{b_{j_2}(y_{j_2})} \cdots \reallywidehat{b_{j_\rho}(y_{j_\rho})} \cdots b_s(y_s) : \nonumber\\
    = \sum_{\rho=1}^{r} \sum_{1\leq i_1 < \cdots < i_\rho \leq r} \sum_{1\leq j_1 < \cdots < j_\rho \leq s} & (-1)^{i_1+\cdots +i_\rho + j_1 + \cdots + j_\rho }(-1)^{(r+1)\rho+\frac{\rho(\rho+1)}2}\nonumber\\
    \cdot & \begin{vmatrix}
    (a_{i_1}, b_{j_1})f(x_{i_1}, y_{j_1}) & \cdots & (a_{i_1}, b_{j_\rho})f(x_{i_1}, y_{j_\rho}) \nonumber\\
    \vdots & & \vdots \nonumber\\
    (a_{i_\rho}, b_{j_1})f(x_{i_\rho}, y_{j_1}) & \cdots & (a_{i_1}, b_{j_\rho})f(x_{i_1}, y_{j_\rho}) 
    \end{vmatrix}\nonumber\\
    \cdot & : a_1(x_1) \cdots \reallywidehat{a_{i_1}(x_{i_1})} \cdots  \reallywidehat{a_{i_\rho}(x_{i_\rho})} \cdots a_r(x_r) a_{r+1}(x_{r+1})\nonumber\\
    & \quad \cdot b_1(y_1) \cdots \reallywidehat{b_{j_1}(y_{j_1})} \cdots \reallywidehat{b_{j_\rho}(y_{j_\rho})} \cdots b_s(y_s) : \label{Prod-Line-1-1} \\
    \quad +\sum_{\rho=1}^{r} \sum_{1\leq i_1 < \cdots < i_\rho \leq r} \sum_{1\leq j_2 < \cdots < j_\rho \leq s} & (-1)^{i_1+\cdots +i_\rho +  j_2 + \cdots + j_\rho }(-1)^{(r+1)\rho+\frac{\rho(\rho+1)}2+\rho}\nonumber\\
    \cdot & \begin{vmatrix}
    (a_{i_1}, a_{r+1})f(x_{i_1}, x_{r+1}) & (a_{i_1}, b_{j_2})f(x_{i_1}, y_{j_2}) & \cdots & (a_{i_1}, b_{j_\rho})f(x_{i_1}, y_{j_\rho}) \\
    \vdots & \vdots & & \vdots \\
    (a_{i_\rho}, a_{r+1})f(x_{i_\rho}, x_{r+1})& (a_{i_\rho}, b_{j_2})f(x_{i_\rho}, y_{j_2}) & \cdots & (a_{i_1}, b_{j_\rho})f(x_{i_1}, y_{j_\rho}) 
    \end{vmatrix}\nonumber\\
    \cdot & : a_1(x_1) \cdots \reallywidehat{a_{i_1}(x_{i_1})} \cdots  \reallywidehat{a_{i_\rho}(x_{i_\rho})} \cdots a_r(x_r) \nonumber\\
    & \quad \cdot b_1(y_1) \cdots \reallywidehat{b_{j_2}(y_{j_2})} \cdots \reallywidehat{b_{j_\rho}(y_{j_\rho})} \cdots b_s(y_s) : \label{Prod-Line-1-2}
\end{align}
We show that (\ref{Prod-Line-1-2}) cancels out with (\ref{Prod-Line-3}). By the induction hypothesis, (\ref{Prod-Line-3}) is expressed as 
\begin{align}
    -\sum_{k=1}^r &(-1)^{r-k} (a_k, a_{r+1})f(x_k, x_{r+1}): a_1(x_1) \cdots \reallywidehat{a_k(x_k)} \cdots a_r(x_r) b_1(y_1)\cdots b_s(y_s):\nonumber\\
    \quad -\sum_{k=1}^r &(-1)^{r-k} (a_k, a_{r+1})f(x_i, x_{r+1}) \sum_{\rho=1}^{r-1}(-1)^{(r-1)\rho+\frac{\rho(\rho+1)}{2}}\nonumber\\
    & \sum_{\zeta=1}^{\rho+1} \sum_{\substack{1\leq i_1< \cdots < i_{\zeta-1} < k\\
    k< i_{\zeta}< \cdots < i_{\rho} \leq r}} \sum_{1\leq j_1 < \cdots < j_\rho \leq s} (-1)^{i_1+ \cdots + i_{\zeta-1} + i_{\zeta}-1 + \cdots + i_\rho-1}(-1)^{j_1 + \cdots + j_\rho} \nonumber\\
    & \qquad \qquad \qquad \qquad \qquad \qquad \quad \cdot 
    \begin{vmatrix}
    (a_{i_1}, b_{j_1})f(x_{i_1}, y_{j_1}) & \cdots & (a_{i_1}, b_{j_\rho})f(x_{i_1}, y_{j_\rho}) \nonumber\\
    \vdots & & \vdots \nonumber\\
    (a_{i_\rho}, b_{j_1})f(x_{i_\rho}, y_{j_1}) & \cdots & (a_{i_1}, b_{j_\rho})f(x_{i_1}, y_{j_\rho}) 
    \end{vmatrix}\nonumber\\
    &\qquad \qquad \qquad \qquad \qquad \qquad \quad \cdot : a_1(x_1) \cdots \reallywidehat{a_{i_1}(x_{i_1})} \cdots  \reallywidehat{a_{i_\rho}(x_{i_\rho})} \cdots a_r(x_r) \nonumber\\
    & \qquad \qquad \qquad \qquad \qquad \qquad \qquad  \cdot b_1(y_1) \cdots \reallywidehat{b_{j_2}(y_{j_2})} \cdots \reallywidehat{b_{j_\rho}(y_{j_\rho})} \cdots b_s(y_s) : 
    \nonumber\\
    = -\sum_{k=1}^r &(-1)^{r-k} (a_k, a_{r+1})f(x_k, x_{r+1}): a_1(x_1) \cdots \reallywidehat{a_k(x_k)} \cdots a_r(x_r) b_1(y_1)\cdots b_s(y_s):\label{Prod-Line-3-1}\\
    \quad -\sum_{k=1}^r & (a_k, a_{r+1})f(x_i, x_{r+1}) \sum_{\rho=1}^{r-1}(-1)^{k+r+r\rho+\frac{\rho(\rho+1)}{2}+\zeta+1}\nonumber\\
    & \sum_{\zeta=1}^{\rho+1} \sum_{\substack{1\leq i_1< \cdots < i_{\zeta-1} < k\\
    k< i_{\zeta}< \cdots < i_{\rho} \leq r}} \sum_{1\leq j_1 < \cdots < j_\rho \leq s} (-1)^{i_1+ \cdots + i_{\zeta-1} + i_{\zeta} + \cdots + i_\rho}(-1)^{j_1 + \cdots + j_\rho} \nonumber\\
    & \qquad \qquad \qquad \qquad \qquad \qquad \quad \cdot 
    \begin{vmatrix}
    (a_{i_1}, b_{j_1})f(x_{i_1}, y_{j_1}) & \cdots & (a_{i_1}, b_{j_\rho})f(x_{i_1}, y_{j_\rho}) \nonumber\\
    \vdots & & \vdots \nonumber\\
    (a_{i_\rho}, b_{j_1})f(x_{i_\rho}, y_{j_1}) & \cdots & (a_{i_1}, b_{j_\rho})f(x_{i_1}, y_{j_\rho}) 
    \end{vmatrix}\nonumber\\
    &\qquad \qquad \qquad \qquad \qquad \qquad \quad \cdot : a_1(x_1) \cdots \reallywidehat{a_{i_1}(x_{i_1})} \cdots  \reallywidehat{a_{i_\rho}(x_{i_\rho})} \cdots a_r(x_r) \nonumber\\
    & \qquad \qquad \qquad \qquad \qquad \qquad \qquad \cdot b_1(y_1) \cdots \reallywidehat{b_{j_2}(y_{j_2})} \cdots \reallywidehat{b_{j_\rho}(y_{j_\rho})} \cdots b_s(y_s) : 
    \label{Prod-Line-3-2}
\end{align}
We turn to (\ref{Prod-Line-1-2}) and look at each individual summand. Write down $\rho=1$ summand of (\ref{Prod-Line-1-2}) and replace $i_1$ by $k$, we have
\begin{align*}
    \sum_{k=1}^r (-1)^{k+1+r+1} (a_k, a_{r+1})f(x_k, x_{r+1}) : a_1(x_1) \cdots \reallywidehat{a_k(x_k)} \cdots a_r(x_r) b_1(y_1)\cdots b_s(y_s):
\end{align*}
which cancels out (\ref{Prod-Line-3-1}). 
For general $\rho\geq 2$ summand in (\ref{Prod-Line-1-2}), we rewrite it as
\begin{align*}
    \sum_{1\leq i_1 < \cdots < i_\rho \leq r} \sum_{1\leq j_2 < \cdots < j_\rho \leq s} & (-1)^{i_1+\cdots +i_\rho  + j_2 + \cdots + j_\rho} 
    (-1)^\rho(-1)^{r\rho+\frac{\rho(\rho+1)}2} \nonumber\\
    \cdot & \sum_{\zeta=1}^\rho (-1)^{\zeta+1}(a_{i_\zeta}, a_{r+1}) f(x_{i_\zeta}, x_{r+1}) \\
    \cdot & \begin{vmatrix}
    (a_{i_1}, b_{j_2})f(x_{i_1}, y_{j_2}) & \cdots & (a_{i_1}, b_{j_\rho})f(x_{i_1}, y_{j_\rho}) \\
    \vdots & & \vdots \\
    \hbox{\sout{$(a_{i_\zeta}, b_{j_2})f(x_{i_\zeta}, y_{j_2})$}} & & \hbox{\sout{$(a_{i_\zeta}, b_{j_\rho})f(x_{i_\zeta}, y_{j_\rho})$}} \\
    \vdots & & \vdots \\
    (a_{i_\rho}, b_{j_2})f(x_{i_\rho}, y_{j_2}) & \cdots & (a_{i_1}, b_{j_\rho})f(x_{i_1}, y_{j_\rho}) 
    \end{vmatrix}\nonumber\\
    \cdot & : a_1(x_1) \cdots \reallywidehat{a_{i_1}(x_{i_1})} \cdots  \reallywidehat{a_{i_\rho}(x_{i_\rho})} \cdots a_r(x_r) \nonumber\\
    & \quad \cdot b_1(y_1) \cdots \reallywidehat{b_{j_1}(y_{j_1})} \cdots \reallywidehat{b_{j_\rho}(y_{j_\rho})} \cdots b_s(y_s) :
\end{align*}
that simplifies as 
\begin{align*}
    \sum_{1\leq i_1 < \cdots < i_\rho \leq r} \sum_{1\leq j_2 < \cdots < j_\rho \leq s} \sum_{\zeta=1}^\rho & (-1)^{i_1+\cdots i_{\zeta-1}+i_{\zeta+1}+\cdots +i_\rho  + j_2 + \cdots + j_\rho} (-1)^{i_\zeta+\rho+r\rho+\frac{\rho(\rho+1)}2+\zeta+1} \nonumber\\
    \cdot &(a_{i_\zeta}, a_{r+1}) f(x_{i_\zeta}, x_{r+1}) \begin{vmatrix}
    (a_{i_1}, b_{j_2})f(x_{i_1}, y_{j_2}) & \cdots & (a_{i_1}, b_{j_\rho})f(x_{i_1}, y_{j_\rho}) \\
    \vdots & & \vdots \\
    \hbox{\sout{$(a_{i_\zeta}, b_{j_2})f(x_{i_\zeta}, y_{j_2})$}} & & \hbox{\sout{$(a_{i_\zeta}, b_{j_\rho})f(x_{i_\zeta}, y_{j_\rho})$}} \\
    \vdots & & \vdots \\
    (a_{i_\rho}, b_{j_2})f(x_{i_\rho}, y_{j_2}) & \cdots & (a_{i_1}, b_{j_\rho})f(x_{i_1}, y_{j_\rho}) 
    \end{vmatrix}\nonumber\\
    \cdot & : a_1(x_1) \cdots \reallywidehat{a_{i_1}(x_{i_1})} \cdots  \reallywidehat{a_{i_\rho}(x_{i_\rho})} \cdots a_r(x_r) \nonumber\\
    & \quad \cdot b_1(y_1) \cdots \reallywidehat{b_{j_1}(y_{j_1})} \cdots \reallywidehat{b_{j_\rho}(y_{j_\rho})} \cdots b_s(y_s) :
\end{align*}
We move the summation of $\zeta$ to the front, then separate the summation of $i_\zeta$, replace $i_\zeta, i_{\zeta+1}, ..., i_\rho$ by $k, i_{\zeta}, ..., i_{\rho-1}$, replace $j_2, ..., j_\rho$ by $j_1, ..., j_{\rho-1}$, finally replace $\rho$ by $\rho+1$, so that the summand is expressed as
\begin{align*}
    \sum_{\zeta=1}^{\rho+1}  \sum_{k=1}^r \sum_{\substack{1\leq i_1 < \cdots < i_{\zeta-1}< k\\k < i_{\zeta}<\cdots < i_{\rho} \leq r}} \sum_{1\leq j_1 < \cdots < j_{\rho} \leq s} &   (-1)^{i_1+\cdots +i_{\rho}  + j_1 + \cdots + j_{\rho}} 
    (-1)^{k+\rho+r\rho+\frac{\rho(\rho+1)}2 + \zeta+1} \nonumber\\
    \cdot & (a_{k}, a_{r+1})f(x_{k}, x_{r+1}) \begin{vmatrix}
    (a_{i_1}, b_{j_1})f(x_{i_1}, y_{j_1}) & \cdots & (a_{i_1}, b_{j_{\rho}})f(x_{i_1}, y_{j_{\rho}}) \\
    \vdots & & \vdots \\
    (a_{i_\rho}, b_{j_1})f(x_{i_\rho}, y_{j_1}) & \cdots & (a_{i_1}, b_{j_\rho})f(x_{i_1}, y_{j_\rho}) 
    \end{vmatrix}\nonumber\\
    \cdot & : a_1(x_1) \cdots \reallywidehat{a_{i_1}(x_{i_1})} \cdots \widehat{a_{i_{\zeta-1}}}(x_{i_{\zeta-1}}) \\
    & \quad \cdots \reallywidehat{a_k(x_k)} \cdots \reallywidehat{a_{i_{\zeta}}(x_{i_\zeta})} \cdots  \reallywidehat{a_{i_\rho}(x_{i_\rho})} \cdots a_r(x_r) \nonumber\\
    & \quad \cdot b_1(y_1) \cdots \reallywidehat{b_{j_1}(y_{j_1})} \cdots \reallywidehat{b_{j_\rho}(y_{j_\rho})} \cdots b_s(y_s) :
\end{align*}
which cancels (\ref{Prod-Line-3-2}). 

Thus what remains is $(\ref{Prod-Fully-Normal})+(\ref{Prod-Line-1-1})+(\ref{Prod-Line-2})$. We want to show that $(\ref{Prod-Line-1-1})+(\ref{Prod-Line-2})$ is 
\begin{align}
    \sum_{\rho=1}^{r+1} \sum_{1\leq i_1 < \cdots < i_\rho \leq r+1} \sum_{1\leq j_1 < \cdots < j_\rho \leq s} & (-1)^{i_1+\cdots +i_\rho + j_1 + \cdots + j_\rho}(-1)^{(r+1)\rho+\frac{\rho(\rho+1)}2}\nonumber\\
    \cdot & \begin{vmatrix}
    (a_{i_1}, b_{j_1})f(x_{i_1}, y_{j_1}) & \cdots & (a_{i_1}, b_{j_\rho})f(x_{i_1}, y_{j_\rho}) \\
    \vdots & & \vdots \nonumber\\
    (a_{i_\rho}, b_{j_1})f(x_{i_\rho}, y_{j_1}) & \cdots & (a_{i_1}, b_{j_\rho})f(x_{i_1}, y_{j_\rho}) 
    \end{vmatrix}\nonumber\\
    \cdot & : a_1(x_1) \cdots \reallywidehat{a_{i_1}(x_{i_1})} \cdots  \reallywidehat{a_{i_\rho}(x_{i_\rho})} \cdots a_{r+1}(x_{r+1})\nonumber\\
    & \quad \cdot b_1(y_1) \cdots \reallywidehat{b_{j_1}(y_{j_1})} \cdots \reallywidehat{b_{j_\rho}(y_{j_\rho})} \cdots b_s(y_s) : \label{Prod-Target}
\end{align}
It is clear that (\ref{Prod-Line-1-1}) accounts for the summands with $i_\rho \leq r$ in (\ref{Prod-Target}). 

We replace the $j$ in (\ref{Prod-Line-2}) by $k$ and apply induction hypothesis: 
\begin{align}
    & \sum_{k=1}^s (-1)^{k-1} (a_{r+1}, b_k)f(x_{r+1}, y_k) : a_1(x_1)\cdots a_r(x_r)\reallywidehat{a_{r+1}(x_{r+1})} b_1(y_1) \cdots \reallywidehat{b_k(y_k)} \cdots b_s(y_s):  
    \nonumber\\ 
     & + \sum_{k=1}^s (-1)^{k-1} (a_{r+1}, b_k)f(x_{r+1}, y_k) \nonumber \\
    & \quad \cdot \sum_{\rho=1}^r \sum_{1\leq i_1 < \cdots < i_\rho \leq r} \sum_{\substack{1\leq j_1 < \cdots < j_{\zeta-1} < k \\
    k < j_{\zeta}< \cdots < j_{\rho} \leq r }} (-1)^{i_1 + \cdots + i_\rho + j_1 + \cdots + j_{\zeta-1} + j_\zeta-1 + \cdots +  j_\rho-1}(-1)^{r\rho+\frac{\rho(\rho+1)}2}\nonumber \\
    & \qquad \qquad \qquad \qquad \qquad \qquad \quad \cdot  \begin{vmatrix}
    (a_{i_1}, b_{j_1})f(x_{i_1}, y_{j_1}) & \cdots & (a_{i_1}, b_{j_\rho})f(x_{i_1}, y_{j_\rho})\\
    \vdots & & \vdots \\
    (a_{i_\rho}, b_{j_1})f(x_{i_\rho}, y_{j_1}) & \cdots & (a_{i_\rho}, b_{j_\rho})f(x_{i_\rho}, y_{j_\rho})
    \end{vmatrix}\nonumber\\
    &  \qquad \qquad \qquad \qquad \qquad \qquad \quad  \cdot : a_1(x_1) \cdots \reallywidehat{a_{i_1}(x_{i_1})}  \cdots  \reallywidehat{a_{i_{\rho}}(x_{i_{\rho}})}\cdots a_r(x_r)  \nonumber\\
    & \qquad \qquad \qquad \qquad \qquad \qquad \qquad  \cdot b_1(y_1) \cdots \reallywidehat{b_{j_1}(y_{j_1})} \cdots \reallywidehat{b_{j_{\zeta-1}}(y_{\zeta-1})}\cdots \reallywidehat{b_k(y_k)} \nonumber\\
    & \qquad \qquad \qquad \qquad \qquad \qquad \qquad   \cdots \reallywidehat{b_{j_{\zeta}}(y_{\zeta})}\cdots \reallywidehat{b_{j_\rho}(y_{j_\rho})} \cdots b_s(y_s):
    \nonumber \\
    = \ & \sum_{k=1}^s (-1)^{k-1} (a_{r+1}, b_k)f(x_{r+1}, y_k) : a_1(x_1)\cdots a_r(x_r)\reallywidehat{a_{r+1}(x_{r+1})} b_1(y_1) \cdots \reallywidehat{b_k(y_k)} \cdots b_s(y_s):  
    \label{Prod-Line-2-1}\\ 
     & + \sum_{k=1}^s (-1)^{k-1} (a_{r+1}, b_k)f(x_{r+1}, y_k) \nonumber \\
    & \quad \cdot \sum_{\rho=1}^r \sum_{1\leq i_1 < \cdots < i_\rho \leq r} \sum_{\substack{1\leq j_1 < \cdots < j_{\zeta-1} < k \\
    k < j_{\zeta}< \cdots < j_{\rho} \leq r }} (-1)^{i_1 + \cdots + i_\rho + j_1 + \cdots + j_{\zeta-1} + j_\zeta + \cdots +  j_\rho}(-1)^{r\rho+\frac{\rho(\rho+1)}2+\rho-\zeta+1}\nonumber \\
    & \qquad \qquad \qquad \qquad \qquad \qquad \quad \cdot  \begin{vmatrix}
    (a_{i_1}, b_{j_1})f(x_{i_1}, y_{j_1}) & \cdots & (a_{i_1}, b_{j_\rho})f(x_{i_1}, y_{j_\rho})\\
    \vdots & & \vdots \\
    (a_{i_\rho}, b_{j_1})f(x_{i_\rho}, y_{j_1}) & \cdots & (a_{i_\rho}, b_{j_\rho})f(x_{i_\rho}, y_{j_\rho})
    \end{vmatrix}\nonumber\\
    &  \qquad \qquad \qquad \qquad \qquad \qquad \quad  \cdot : a_1(x_1) \cdots \reallywidehat{a_{i_1}(x_{i_1})}  \cdots  \reallywidehat{a_{i_{\rho}}(x_{i_{\rho}})}\cdots a_r(x_r)  \nonumber\\
    & \qquad \qquad \qquad \qquad \qquad \qquad \qquad  \cdot b_1(y_1) \cdots \reallywidehat{b_{j_1}(y_{j_1})} \cdots \reallywidehat{b_{j_{\zeta-1}}(y_{\zeta-1})}\cdots \reallywidehat{b_k(y_k)} \nonumber\\
    & \qquad \qquad \qquad \qquad \qquad \qquad \qquad   \cdots \reallywidehat{b_{j_{\zeta}}(y_{\zeta})}\cdots \reallywidehat{b_{j_\rho}(y_{j_\rho})} \cdots b_s(y_s):
 \label{Prod-Line-2-2} 
\end{align}
It is clear that (\ref{Prod-Line-2-1}) accounts for the summand of (\ref{Prod-Target}) with $i_\rho = r+1$ and $\rho=1$ in (\ref{Prod-Target}). 
To see that (\ref{Prod-Line-2-2}) accounts for the summands of (\ref{Prod-Target}) with $i_\rho=r+1$ and $\rho\geq 2$, we rewrite the sum of related summands in (\ref{Prod-Target}) as
\begin{align*}
    \sum_{\rho=2}^{r+1} \sum_{1\leq i_1 < \cdots < i_{\rho-1} \leq r} \sum_{1\leq j_1 < \cdots < j_\rho \leq s} & (-1)^{i_1+\cdots +i_{\rho-1}+r+1 + j_1 + \cdots + j_\rho}(-1)^{(r+1)\rho + \frac{\rho(\rho+1)}2 }\nonumber\\
    \cdot & \begin{vmatrix}
    (a_{i_1}, b_{j_1})f(x_{i_1}, y_{j_1}) & \cdots & (a_{i_1}, b_{j_\rho})f(x_{i_1}, y_{j_\rho}) \\
    \vdots & & \vdots \nonumber\\
    (a_{i_{\rho-1}}, b_{j_1})f(x_{i_{\rho-1}}, y_{j_1}) & \cdots & (a_{i_{\rho-1}}, b_{j_\rho})f(x_{i_1}, y_{j_\rho})\\
    (a_{r+1}, b_{j_1})f(x_{r+1}, y_{j_1}) & \cdots & (a_{r+1}, b_{j_\rho})f(x_{r+1}, y_{j_\rho}) 
    \end{vmatrix}\nonumber\\
    \cdot & : a_1(x_1) \cdots \reallywidehat{a_{i_1}(x_{i_1})} \cdots  \reallywidehat{a_{i_{\rho-1}}(x_{i_{\rho-1}})}\cdots a_{r}(x_{r}) \reallywidehat{a_{r+1}(x_{r+1})}\nonumber\\
    & \quad \cdot b_1(y_1) \cdots \reallywidehat{b_{j_1}(y_{j_1})} \cdots \reallywidehat{b_{j_\rho}(y_{j_\rho})} \cdots b_s(y_s) :\\
    =\sum_{\rho=2}^{r+1} \sum_{1\leq i_1 < \cdots < i_{\rho-1} \leq r} \sum_{1\leq j_1 < \cdots < j_\rho \leq s} & (-1)^{i_1+\cdots +i_{\rho-1} + j_1 + \cdots + j_\rho}(-1)^{(r+1)(\rho+1) + \frac{\rho(\rho+1)}2}\nonumber\\
    \cdot\sum_{\zeta=1}^\rho &  (-1)^{\rho+\zeta}(a_{r+1}, b_{i_\zeta})f(x_{r+1}, y_{i_\zeta}) \\
    \cdot& \begin{vmatrix}
    (a_{i_1}, b_{j_1})f(x_{i_1}, y_{j_1}) & \cdots & (a_{i_1}, b_{j_\rho})f(x_{i_1}, y_{j_\rho}) \\
    \vdots & & \vdots \nonumber\\
    (a_{i_{\rho-1}}, b_{j_1})f(x_{i_{\rho-1}}, y_{j_1}) & \cdots & (a_{i_{\rho-1}}, b_{j_\rho})f(x_{i_1}, y_{j_\rho})
    \end{vmatrix}\nonumber\\
    \cdot & : a_1(x_1) \cdots \reallywidehat{a_{i_1}(x_{i_1})} \cdots  \reallywidehat{a_{i_{\rho-1}}(x_{i_{\rho-1}})}\cdots a_{r}(x_{r})\reallywidehat{a_{r+1}(x_{r+1})}\nonumber\\
    & \quad \cdot b_1(y_1) \cdots \reallywidehat{b_{j_1}(y_{j_1})} \cdots \reallywidehat{b_{j_\rho}(y_{j_\rho})} \cdots b_s(y_s) :
\end{align*}
We move the summation of $\zeta$ to the front, then separate the summation of $j_\zeta$, replace $j_\zeta, j_{\zeta+1}, ..., j_\rho$ by $k, j_{\zeta}, ..., j_{\rho-1}$, finally replace $\rho$ by $\rho+1$, to get
\begin{align*}
     \sum_{\rho=1}^{r}\sum_{\zeta=1}^\rho \sum_{k=1}^r(-1)^{\rho+1+\zeta+r+1}(-1)^k&(a_{r+1}, b_{k})f(x_{r+1}, y_{k})\\
    \sum_{1\leq i_1 < \cdots < i_{\rho-1} \leq r} \sum_{\substack{1\leq j_1 < \cdots < j_{\zeta-1} < k \\
    k < j_{\zeta} < \cdots < j_\rho \leq s}} & (-1)^{i_1+\cdots +i_{\rho} + j_1 + \cdots + j_\rho}(-1)^{(r+1)(\rho+1) + \frac{(\rho+1)(\rho+2)}2}\nonumber\\
    \cdot& \begin{vmatrix}
    (a_{i_1}, b_{j_1})f(x_{i_1}, y_{j_1}) & \cdots & (a_{i_1}, b_{j_\rho})f(x_{i_1}, y_{j_\rho}) \\
    \vdots & & \vdots \nonumber\\
    (a_{i_{\rho}}, b_{j_1})f(x_{i_{\rho}}, y_{j_1}) & \cdots & (a_{i_{\rho}}, b_{j_\rho})f(x_{i_1}, y_{j_\rho})
    \end{vmatrix}\nonumber\\
    \cdot & : a_1(x_1) \cdots \reallywidehat{a_{i_1}(x_{i_1})} \cdots  \reallywidehat{a_{i_{\rho}}(x_{i_{\rho}})}\cdots a_{r}(x_{r})\reallywidehat{a_{r+1}(x_{r+1})}\nonumber\\
    & \cdot b_1(y_1) \cdots \reallywidehat{b_{j_1}(y_{j_1})} \cdots \widehat{b_{j_{\zeta-1}}}(y_{\zeta-1})\cdots \reallywidehat{b_k(y_k)} \\& \quad \cdots \reallywidehat{b_{j_{\zeta}}(y_{\zeta})}\cdots \reallywidehat{b_{j_\rho}(y_{j_\rho})} \cdots b_s(y_s) :\\
    =  \sum_{\rho=1}^{r}\sum_{\zeta=1}^\rho \sum_{k=1}^r(-1)^{k+1}(a_{r+1}, b_{k}) f(x_{r+1}&, y_{k}) \\
    \sum_{1\leq i_1 < \cdots < i_{\rho-1} \leq r} \sum_{\substack{1\leq j_1 < \cdots < j_{\zeta-1} < k \\
    k < j_{\zeta} < \cdots < j_\rho \leq s}} & (-1)^{i_1+\cdots +i_{\rho} + j_1 + \cdots + j_\rho}(-1)^{r\rho+\rho + \frac{(\rho+1)\rho}2+\zeta+1}\nonumber\\
    \cdot& \begin{vmatrix}
    (a_{i_1}, b_{j_1})f(x_{i_1}, y_{j_1}) & \cdots & (a_{i_1}, b_{j_\rho})f(x_{i_1}, y_{j_\rho}) \\
    \vdots & & \vdots \nonumber\\
    (a_{i_{\rho}}, b_{j_1})f(x_{i_{\rho}}, y_{j_1}) & \cdots & (a_{i_{\rho}}, b_{j_\rho})f(x_{i_1}, y_{j_\rho})
    \end{vmatrix}\nonumber\\
    \cdot & : a_1(x_1) \cdots \reallywidehat{a_{i_1}(x_{i_1})} \cdots  \reallywidehat{a_{i_{\rho}}(x_{i_{\rho}})}\cdots a_{r}(x_{r})\reallywidehat{a_{r+1}(x_{r+1})}\nonumber\\
    & \cdot b_1(y_1) \cdots \reallywidehat{b_{j_1}(y_{j_1})} \cdots \widehat{b_{j_{\zeta-1}}}(y_{\zeta-1})\cdots \reallywidehat{b_k(y_k)} \\
    & \quad \cdots \reallywidehat{b_{j_{\zeta}}(y_{\zeta})}\cdots \reallywidehat{b_{j_\rho}(y_{j_\rho})} \cdots b_s(y_s) :
\end{align*}
This coincides with (\ref{Prod-Line-2-2}). 
So the conclusion holds under the condition $r+1\leq s$. 

For $r\geq s$, the argument is almost verbatim. The only necessary modification is on the upper limit of the summation with respect to $\rho$. The upper limit should be modified from $r$ to $s$. Other steps are completely identical. We should not repeat the details. 
\end{proof}

\begin{rema}
What we have proved can be viewed as a generalization of Wick's theorem in quantum field theory, expressing a product of two normal-ordered products in terms of other normal-ordered products. Theorem \ref{Product-r-s-thm} shows that Wick's theorem holds without assuming any relations among the creation operators. 
\end{rema}

\subsection{Product and iterate formulas}

\begin{cor}\label{Product-thm}
    For $r, s\in \N, a_1, ..., a_r, b_1, ..., b_s\in \h, m_1, ..., m_r, n_1, ..., n_s\in \N$
    \begin{align}
        & Y(a_1(-m_1-1/2)\cdots a_r(-m_r-1/2)\one, x)Y(b_1(-n_1-1/2)\cdots b_s(-n_s-1/2)\one, y) \nonumber\\
        = \ & :a_1^{(m_1)}(x)\cdots a_r^{(m_r)}(x)b_1^{(n_1)}(y)\cdots b_s^{(n_s)}(y):\nonumber \\
        & + \sum_{\rho=1}^{\min(r,s)}\sum_{\substack{ 1 \leq i_1 < \cdots < i_\rho\leq r\\ 1 \leq j_1 < \cdots < j_\rho \leq s}}(-1)^{i_1+\cdots + i_\rho + j_1+\cdots + j_\rho }(-1)^{r\rho + \rho(\rho+1)/2}\nonumber \\
        & \hspace{9.7 em}\cdot \begin{vmatrix}
            (a_{i_1}, b_{j_1})f_{m_{i_1}n_{j_1}}(x, y) & \cdots & (a_{i_1}, b_{j_\rho})f_{m_{i_1}n_{j_\rho}}(x, y)\\
            \vdots & & \vdots\nonumber  \\
            (a_{i_\rho}, b_{j_1})f_{m_{i_\rho}n_{j_1}}(x, y) & \cdots & (a_{i_\rho}, b_{j_\rho})f_{m_{i_\rho}n_{j_\rho}}(x, y)
        \end{vmatrix}\nonumber \\
        & \hspace{10 em}\cdot :a_{1}^{(m_1)}(x) \cdots \reallywidehat{a_{i_1}^{(m_{i_1})}(x)} \cdots \reallywidehat{a_{i_\rho}^{(m_{i_\rho})}(x)} \cdots a_{r}^{(m_r)}(x)\nonumber  \\
        & \hspace{11 em}\cdot b_{1}^{(n_1)}(y) \cdots \reallywidehat{b_{j_1}^{(n_{j_1})}(y)} \cdots \reallywidehat{b_{j_\rho}^{(n_{j_\rho})}(y)} \cdots b_{s}^{(n_s)}(y):\label{Product-formula}
    \end{align}
\end{cor}

\begin{proof}
    Directly substitute $x_1 = x, ..., x_r = r, y_1=y, ..., y_s = y$ in Theorem \ref{Product-r-s-thm} to see the conclusion.  
\end{proof}

\begin{cor}
    For $r, s\in \N, a_1, ..., a_r, b_1, ..., b_s\in \h, m_1, ..., m_r, n_1, ..., n_s\in \N$
    \begin{align*}
        & Y(a_1(-m_1-1/2)\cdots a_r(-m_r-1/2)\one, x)b_1(-n_1-1/2)\cdots b_s(-n_s-1/2)\one \\
        = \ & a_1^{(m_1)}(x)^-\cdots a_r^{(m_r)}(x)^-b_1(-n_1-1/2)\cdots b_s(-n_s-1/2)\one\nonumber \\
        & + \sum_{\rho=1}^{\min(r,s)}\sum_{\substack{ 1 \leq i_1 < \cdots < i_\rho\leq r\\ 1 \leq j_1 < \cdots < j_\rho \leq s}}(-1)^{i_1+\cdots + i_\rho + j_1+\cdots + j_\rho }(-1)^{r\rho + \rho(\rho+1)/2}\nonumber \\
        & \hspace{9.7 em}\cdot \begin{vmatrix}
            (a_{i_1}, b_{j_1})(x^{-n_{j_1}-1})^{(m_{i_1})} & \cdots & (a_{i_1}, b_{j_\rho})(x^{-n_{j_\rho}-1})^{(m_{i_1})}\\
            \vdots & & \vdots\nonumber  \\
            (a_{i_\rho}, b_{j_1})(x^{-n_{j_1}-1})^{(m_{i_\rho})} & \cdots & (a_{i_\rho}, b_{j_\rho})(x^{-n_{j_\rho}-1})^{(m_{i_\rho})} 
        \end{vmatrix}\nonumber \\
        & \hspace{10 em}\cdot a_{1}^{(m_1)}(x)^- \cdots \reallywidehat{a_{i_1}^{(m_{i_1})}(x)^-} \cdots \reallywidehat{a_{i_\rho}^{(m_{i_\rho})}(x)^-} \cdots a_{r}^{(m_r)}(x)^-\nonumber  \\
        & \hspace{10 em}\cdot b_{1}(-n_1/1-2) \cdots \reallywidehat{b_{j_1}(-n_{j_1}-1/2)} \cdots \reallywidehat{b_{j_\rho}(-n_{j_\rho}-1/2)} \cdots b_{s}(-n_s-1/2)\one
    \end{align*}
\end{cor}

\begin{proof}
    It follows from acting the formula (\ref{Product-r-s-formula}) in Theorem \ref{Product-r-s-thm} on the vacuum $\one$, then evaluate $y_1 = 0, ..., y_s = 0$. Note that 
    \begin{align*}
        f_{mn}(x,0)= \binom{-n-1}{m} (x-0)^{-n-m-1}= \binom{-n-1}{m}x^{-n-1-m} = (x^{-n-1})^{(m)}. 
    \end{align*}
    Note also that 
    \begin{align*}
        & :a_1(x_1)\cdots a_r(x_r)b_1(-n_1-1/2)\cdots b_s(-n_s-1/2):\one \\
        = \ & a_1(x_1)^-\cdots a_r(x_r)^-b_1(-n_1-1/2)\cdots b_s(-n_s-1/2)\one
    \end{align*}
    since any positive mode in the normal ordering annihilates the vaccum. 
    The conclusion easily follows from these observations. 
\end{proof}

\begin{cor}\label{Iterate-thm}
    For $r, s\in \N, a_1, ..., a_r, b_1, ..., b_s\in \h, m_1, ..., m_r, n_1, ..., n_s\in \N$
    \begin{align}
        & Y(Y(a_1(-m_1-1/2)\cdots a_r(-m_r-1/2)\one, x)b_1(-n_1-1/2)\cdots b_s(-n_s-1/2)\one, y)\nonumber \\
        = \ & :a_1^{(m_1)}(y+x)\cdots a_r^{(m_r)}(y+x)b_1^{(n_1)}(y)\cdots b_s^{(n_s)}(y):\nonumber \\
        & + \sum_{\rho=1}^{\min(r,s)}\sum_{\substack{ 1 \leq i_1 < \cdots < i_\rho\leq r\\ 1 \leq j_1 < \cdots < j_\rho \leq s}}(-1)^{i_1+\cdots + i_\rho + j_1+\cdots + j_\rho }(-1)^{r\rho + \rho(\rho+1)/2}\nonumber \\
        & \hspace{9.7 em}\cdot \begin{vmatrix}
            (a_{i_1}, b_{j_1})(x^{-n_{j_1}-1})^{(m_{i_1})} & \cdots & (a_{i_1}, b_{j_\rho})(x^{-n_{j_\rho}-1})^{(m_{i_1})}\\
            \vdots & & \vdots\nonumber  \\
            (a_{i_\rho}, b_{j_1})(x^{-n_{j_1}-1})^{(m_{i_\rho})} & \cdots & (a_{i_\rho}, b_{j_\rho})(x^{-n_{j_\rho}-1})^{(m_{i_\rho})} 
        \end{vmatrix}\nonumber \\
        & \hspace{10 em}\cdot :a_{1}^{(m_1)}(y+x) \cdots \reallywidehat{a_{i_1}^{(m_{i_1})}(y+x)} \cdots \reallywidehat{a_{i_\rho}^{(m_{i_\rho})}(y+x)} \cdots a_{r}^{(m_r)}(y+x)\nonumber  \\
        & \hspace{11 em}\cdot b_{1}^{(n_1)}(y) \cdots \reallywidehat{b_{j_1}^{(n_{j_1})}(y)} \cdots \reallywidehat{b_{j_\rho}^{(n_{j_\rho})}(y)} \cdots b_{s}^{(n_s)}(y):\label{Iterate-formula}
    \end{align}
\end{cor}

\begin{proof}
    We first compute
    \begin{align}
        & Y(a_1^{(m_1)}(x)^- \cdots a_r^{(m_r)}(x)^- b_1(-n_1-1/2)\cdots b_s(-n_s-1/2)\one, y)\label{Iterate-1} \\
        = \ & \sum_{i_1, ..., i_r\geq 0} Y(a_1(-i_1-1/2)\cdots a_r(-i_r-1/2)b_1(-n_1-1/2)\cdots b_s(-n_s-1/2)\one, y)(x^{i_1})^{(m_1)}\cdots (x^{i_r})^{(m_r)}\nonumber\\
        = \ & \sum_{i_1, ..., i_r\geq 0} :a_1^{(i_1)}(y)\cdots  a_r^{(i_r)}(y)b_1^{(n_1)}(y)\cdots b_s^{(n_s)}(y):(x^{i_1})^{(m_1)}\cdots (x^{i_r})^{(m_r)}\nonumber\\
        = \ & :\frac{1}{m_1!}\frac{\partial^{m_1}}{\partial x^{m_1}}\left(\sum_{i_1\geq 0} \frac{x^{i_1}}{i_1!}\frac{\partial^{i_1}}{\partial y^{i_1}}a_1(y)\right) \cdots \frac{1}{m_r!}\frac{\partial^{m_r}}{\partial x^{m_r}}\left(\sum_{i_r\geq 0} \frac{x^{i_r}}{i_r!}\frac{\partial^{i_r}}{\partial y^{i_r}}a_r(y)\right) b_1^{(n_1)}(y)\cdots b_s^{(n_s)}(y):\nonumber
    \end{align}
    By the formal Taylor theorem, 
    $$\sum_{i\geq 0} \frac{x^i}{i_1!}\frac{\partial}{\partial x}a(y) = a(y+x), $$
    thus
    \begin{align*}
        (\ref{Iterate-1}) = \ &  :\frac{1}{m_1!}\frac{\partial^{m_1}}{\partial x^{m_1}}\left(a_1(y+x)\right) \cdots \frac{1}{m_r!}\frac{\partial^{m_r}}{\partial x^{m_r}}\left(a_r(y+x)\right)b_1^{(n_1)}(y)\cdots b_s^{(n_s)}(y):    
    \end{align*}
    Noticing from chain rule that 
    $$\frac{\partial }{\partial x} = \frac{\partial}{\partial(y+x)}\cdot \frac{\partial(y+x)}{\partial x} = \frac{\partial}{\partial(y+x)}, $$
    we conclude that 
    \begin{align*}
        (\ref{Iterate-1}) = \ &  :a_1^{(m_1)}(y+x) \cdots a_r^{(m_r)}(y+x) b_1^{(n_1)}(y)\cdots b_s^{(n_s)}(y):
    \end{align*}
    By variation of the choices of $a_1, ..., a_r$ and $m_1, ..., m_r$, we may conclude the proof of the corollary. 
\end{proof}

\subsection{Proof of weak associativity}
We finish the proof of Theorem \ref{main} by showing the following proposition
\begin{prop}
    Let $u = a_1(-m_1-1/2)\cdots a_r(-m_r-1/2)\one, v = b_1(-n_1-1/2) \cdots b_s(-n_s-1/2)\one$ for any $a_1, ..., a_r, b_1, ..., b_s\in \h, m_1, ..., m_r, n_1, ..., n_s\in \N$, then for every homogeneous $w\in W$, the identity
    $$(x+y)^P Y(u, x)Y(v, y)w = (x+y)^P Y(Y(u, x)v, y)w$$
    holds, where
    $$P = \wt(w) + m_1 + \cdots + m_r + r$$
    that is independent of $v$. 
\end{prop}
\begin{proof}
    Notice that 
    $$f_{mn}(x+y, y) = \binom{-n-1}m (x+y-y)^{-n-m-1} = \binom{-n-1}m x^{-n-m-1} = (x^{-n-1})^{(m)}. $$
    Then from the product formula (\ref{Product-formula}) in Corollary \ref{Product-thm}, we see that 
    \begin{align}
        & Y(a_1(-m_1-1/2)\cdots a_r(-m_r-1/2)\one, x+y)Y(b_1(-n_1-1/2)\cdots b_s(-n_s-1/2)\one,y) \nonumber\\
        = \ & :a_1^{(m_1)}(x+y)\cdots a_r^{(m_r)}(x+y)b_1^{(n_1)}(y)\cdots b_s^{(n_s)}(y):\nonumber \\
        & + \sum_{\rho=1}^{\min(r,s)}\sum_{\substack{ 1 \leq i_1 < \cdots < i_\rho\leq r\\ 1 \leq j_1 < \cdots < j_\rho \leq s}}(-1)^{i_1+\cdots + i_\rho + j_1+\cdots + j_\rho }(-1)^{r\rho + \rho(\rho+1)/2}\nonumber \\
        & \hspace{9.7 em}\cdot \begin{vmatrix}
            (a_{i_1}, b_{j_1})(x^{-n_{j_1}-1})^{(m_{i_1})} & \cdots & (a_{i_1}, b_{j_\rho})(x^{-n_{j_\rho}-1})^{(m_{i_1})}\\
            \vdots & & \vdots\nonumber  \\
            (a_{i_\rho}, b_{j_1})(x^{-n_{j_1}-1})^{(m_{i_\rho})} & \cdots & (a_{i_\rho}, b_{j_\rho})(x^{-n_{j_\rho}-1})^{(m_{i_\rho})} 
        \end{vmatrix}\nonumber \\
        & \hspace{10 em}\cdot :a_{1}^{(m_1)}(x+y) \cdots \reallywidehat{a_{i_1}^{(m_{i_1})}(x+y)} \cdots \reallywidehat{a_{i_\rho}^{(m_{i_\rho})}(x+y)} \cdots a_{r}^{(m_r)}(x+y)\nonumber  \\
        & \hspace{11 em}\cdot b_{1}^{(n_1)}(y) \cdots \reallywidehat{b_{j_1}^{(n_{j_1})}(y)} \cdots \reallywidehat{b_{j_\rho}^{(n_{j_\rho})}(y)} \cdots b_{s}^{(n_s)}(y):\label{LHS-weak-assoc}
    \end{align}
    which differs from the iterate formula (\ref{Iterate-formula}) by transposing $x$ and $y$ in the arguments of series $a_1^{(m_1)}, ..., a_r^{(m_r)}$. If we act both (\ref{LHS-weak-assoc}) and (\ref{Iterate-1}) on any homogeneous element $w\in W$ and multiply them with $(x+y)^P$, where 
    $$P = \wt(w) + m_1+1 + \cdots + m_r + 1 $$
    that is independent of $b_1, ..., b_s$ and $n_1, ..., n_s$, then the two resulting series do not contain negative powers of $(x+y)$. So we may freely interchange $x$ and $y$. Consequently, the resulting series are equal. 
\end{proof}

\section{The operator $\exp(\Delta(x))$ on $V$}

To prepare for the construction the $\Z_2$-twisted $V$-module, using the idea from \cite{FLM} and \cite{FFR}, we need an operator $\exp(\Delta(x))$ on $V$. In this section we will study some properties of the operator.

\subsection{Definition of $\Delta(x)$} Recall that $\h = \mathfrak{p} \oplus \mathfrak{q}$ is polarized as a direct sum of maximal isotropic subspaces. Let $e_1, ..., e_M$ be a basis of $\mathfrak{p}$, $\bar e_1, ..., \bar e_M$ be a basis of $\mathfrak{q}$ such that 
$$(e_i, \bar e_j) = \delta_{ij}, (e_i, e_j) = (\bar e_i, \bar e_j) = 0, i,j = 1, ..., M. $$
Define
$$\Delta(x) = \sum_{i=1}^M\sum_{m,n \geq 0}C_{mn} e_i(m+1/2) \bar e_i(n+1/2) x^{-m-n-1}. $$
To make sure that $\Delta(x)$ is independent of the choice of polarizations and basis, from the anti-commutativity of the positive modes, it is necessary and sufficient that $C_{mn} = - C_{nm}$. 

\begin{prop}\label{Delta(x)-Properties}
    \begin{align*}
    (1) \quad &[\Delta(x), a(-n-1/2)] = \sum_{m\geq 0} C_{mn} a(m+1/2) x^{-m-n-1}. \\
    (2) \quad & \Delta(x) b_1(-n_1-1/2) \cdots b_s(-n_s-1/2) \one \\
    & = \sum_{1\leq p < q \leq s} (-1)^{p+q} C_{n_p n_q} (b_p, b_q) x^{-n_p-n_q-1} \\
    & \qquad \qquad \quad \cdot b_1(-n_1-1/2) \cdots \reallywidehat{b_p(-n_p-1/2)} \cdots \reallywidehat{b_q(-n_q-1/2)}\cdots b_s(-n_s-1/2)\one.
\end{align*}
\end{prop}

\begin{proof}
    \begin{enumerate}
        \item We compute as follows
        \begin{align*}
            & \Delta(x) a(-n-1/2) \\
            = \ & \sum_{i=1}^M\sum_{m,n_1 \geq 0}C_{mn_1} x^{-m-n_1-1} e_i(m+1/2) \bar e_i(n_1+1/2) a(-n-1/2)\\
            = \ & \sum_{i=1}^M\sum_{m,n_1 \geq 0}C_{mn_1} x^{-m-n_1-1}  a(-n-1/2) e_i(m+1/2) \bar e_i(n_1+1/2)\\
            & + \sum_{i=1}^M\sum_{m,n_1 \geq 0}C_{mn_1} x^{-m-n_1-1}   e_i(m+1/2) \{\bar e_i(n_1+1/2), a(-n-1/2)\}\\
            & -  \sum_{i=1}^M\sum_{m,n_1 \geq 0}C_{mn_1} x^{-m-n_1-1}   \{e_i(m+1/2), a(-n-1/2)\} \bar e_i(n_1+1/2)\\
            = \ & a(-n-1/2)\Delta(x) + \sum_{i=1}^M\sum_{m\geq 0}C_{mn} x^{-m-n-1}   e_i(m+1/2) (a, \bar e_i)\\
            & -  \sum_{i=1}^M\sum_{m\geq 0}C_{nm} x^{-n-m-1}   (a, e_i) \bar e_i(m+1/2)
        \end{align*}
        where we replaced $n_1$ by $m$ in the third sum. Using the fact that $C_{mn}=-C_{nm}$, we see that 
        \begin{align*}
            [\Delta(x), a(-n-1/2)] = \ & \sum_{m\geq 0} C_{mn} \sum_{i=1}^M \left((a, \bar e_i)e_i(m+1/2) + (a, e_i) \bar e_i(m+1/2)\right) x^{-m-n-1}\\
            = \ & C_{mn} a(m+1/2) x^{-m-n-1}. 
        \end{align*}
        \item The result clearly holds when $s=0$ and $s=1$. Assuming the result for all smaller $s$, we compute as follows 
        \begin{align}
            & \Delta(x) b_1(-n_1-1/2) \cdots b_s(-n_s-1/2) \one \nonumber\\
            = \ & [\Delta(x), b_1(-n_1-1/2)] b_2(-n_2-1/2)\cdots b_s(-n_s-1/2)\one \nonumber\\
            & + b_1(-n_1-1/2) \Delta(x)b_2(-n_2-1/2)\cdots b_s(-n_s-1/2)\one \nonumber\\
            = \ & \sum_{m\geq 0} C_{mn_1}b_1(m+1/2) x^{-m-n_1-1} b_2(-n_2-1/2)\cdots b_s(-n_s-1/2)\one\label{Delta(x)-1}\\
            & + b_1(-n_1-1/2) \sum_{2\leq p<q \leq s}(-1)^{p-1+q-1}C_{n_pn_q}(b_p, b_q) x^{-n_p-n_q-1} \nonumber\\
            & \qquad \qquad \qquad \qquad \qquad \quad \cdot b_2(-n_2-1/2) \cdots \reallywidehat{b_p(-n_p-1/2)} \cdots \reallywidehat{b_q(-n_q-1/2)}\cdots b_s(-n_s-1/2)\one \label{Delta(x)-2}
        \end{align}
        We focus on (\ref{Delta(x)-1}). 
        \begin{align}
            (\ref{Delta(x)-1}) = \ & \sum_{m\geq 0} \sum_{q = 2}^s C_{mn_1} (-1)^{q-2} b_2(-n_2-1/2)\cdots\{b_1(m+1/2), b_q(-n_q-1/2)\} b_s(-n_s-1/2)\one  x^{-m-n_1-1}\nonumber \\
            = \ &  \sum_{q = 2}^sC_{n_qn_1}  (-1)^{q} (b_1, b_q) b_2(-n_2-1/2)\cdots\reallywidehat{b_q(-n_q-1/2)} b_s(-n_s-1/2)\one  x^{-m-n_1-1} \nonumber\\
            = \ & \sum_{q = 2}^s (-1)^{q-1} C_{n_1n_q} (b_1, b_q) b_2(-n_2-1/2)\cdots\reallywidehat{b_q(-n_q-1/2)} b_s(-n_s-1/2)\one  x^{-m-n_1-1} \label{Delta(x)-3}
        \end{align}
        The conclusion then follows by combining (\ref{Delta(x)-2}) and (\ref{Delta(x)-3}). 
    \end{enumerate}
\end{proof}

\subsection{The powers and the exponential of $\Delta(x)$}

\begin{nota}\label{Index-Notation-1}
    To avoid the clumsy iterated subscripts, for $a_1, ...., a_r\in \h$, $m_1, ..., m_r\in \N$ and $1\leq i_1< ...< i_{k}\leq r$, we introduce the following notations. 
    \begin{enumerate}
        \item For $1\leq p < q \leq k$, the notation $$\langle i_p, i_q \rangle =(a_{i_p}, a_{i_q}) C_{m_{i_p}m_{i_q}}$$ 
        stands for a number in $\C$. 
        \item The notation 
        $$S_{a_1,... a_r}^{m_1, ..., m_r}(i_1, ..., i_k) = a_{i_1}(-m_{i_1}-1/2)\cdots a_{i_k}(-m_{i_k}-1/2)\one$$
        stands for an element in $V$
    \end{enumerate}
    These notations will be used in the proofs and in the computations. We will not use them in the statements of the definitions, lemmas, propositions and theorems. 
    
\end{nota}

\begin{defn}
    Let $a_1, ..., a_r\in \h, m_1, ..., m_r\in \N$. Then for each $t\in \Z_+$ increasing sequence $1\leq i_1 < \cdots < i_{2t}$, we define the \textit{total contraction number} $T_{a_1, ..., a_r}^{m_1, ..., m_r}(i_1, ..., i_{2t})$ recursively by
    \begin{align*}
    T_{a_1,...,a_r}^{m_1,...,m_r}(i_1, i_2) &= (a_{i_1},a_{i_2})C_{m_{i_1}m_{i_2}}\\
    T_{a_1,...,a_r}^{m_1,...,m_r}(i_1,..., i_{2t}) &= \sum_{k=2}^{2t} (-1)^{k}(a_{i_1}, a_{i_k}) C_{m_{i_1}m_{i_k}} T^{m_1\cdots m_r}_{a_1,...,a_r}(\widehat{i_1}, i_2, ..., \widehat{i_k}, i_{k+1}, ..., i_{2t}). 
\end{align*}
For example, 
\begin{align*}
    T_{a_1,...,a_r}^{m_1,...,m_r}(i_1, i_2,i_3,i_4)&=(a_{i_1},a_{i_2})C_{m_{i_1}m_{i_2}}(a_{i_3},a_{i_4})C_{m_{i_3}m_{i_4}}-(a_{i_1},a_{i_3})C_{m_{i_1}m_{i_3}}(a_{i_2},a_{i_4})C_{m_{i_2}m_{i_4}}\\
    & \quad +(a_{i_1},a_{i_4})C_{m_{i_1}m_{i_4}}(a_{i_2},a_{i_3})C_{m_{i_2}m_{i_3}}
\end{align*}
\end{defn}

\begin{prop}\label{Contraction-Number-t-recursion}
    The total contraction number $T_{a_1,...,a_r}^{m_1,...,m_r}(i_1, ..., i_{2t})$ satisfies the following alternative recursion:
    \begin{align*}
        T_{a_1,...,a_r}^{m_1,...,m_r}(i_1, ..., i_{2t}) = \frac 1 t \sum_{1\leq \alpha < \beta \leq 2t} (-1)^{\alpha+\beta-1} (a_{i_\alpha}, a_{i_\beta}) C_{m_{i_\alpha}m_{i_\beta}}T_{a_1,...,a_r}^{m_1,...,m_r}(i_1, ..., \widehat{i_\alpha}, ..., \widehat{i_\beta},...,i_{2t})
    \end{align*}
\end{prop}

\begin{proof}
    With Notation \ref{Index-Notation-1}, the recursion defining $T_{a_1,...,a_r}^{m_1, ..., m_r}(i_1, ..., i_{2t})$ is expressed as 
    \begin{align*}
        T_{a_1,...,a_r}^{m_1, ..., m_r}(i_1, i_2) &= \langle i_1, i_2\rangle\\
        T_{a_1,...,a_r}^{m_1, ..., m_r}(i_1, ..., i_{2t}) &= \sum_{k=2}^{2t} (-1)^k \langle i_1, i_k\rangle T_{a_1,...,a_r}^{m_1, ..., m_r}(\widehat{i_1}, i_2, ..., \widehat{i_k},..., i_{2t})
    \end{align*}
    We argue by induction. The conclusion clearly holds when $t=1$. Assume the conclusion for all smaller $t$. For every fixed $k=2, ..., 2t$, we define the series
    $$j_1^{(k)} = i_2, ..., j_{k-1}^{(k)} = i_{k-1}, j_1^{(k)}(k-1) = i_{k+1}, ..., j_1^{(k)}(2t-2) = i_{2t}$$
    Then by definition and the induction hypothesis,
    \begin{align*}
        & T_{a_1,...,a_r}^{m_1, ..., m_r}(i_1, ..., i_{2t}) = \sum_{k=2}^{2t} (-1)^k \langle i_1, i_k\rangle T_{a_1,...,a_r}^{m_1, ..., m_r}(j_1^{(k)}, ..., j_{2t-2}^{(k)}) \\
        = \ & \frac 1 {t-1} \sum_{k=2}^{2t} (-1)^k \langle i_1, i_k\rangle \sum_{1\leq \alpha < \beta \leq 2t-2} (-1)^{\alpha+\beta-1}\langle j_\alpha^{(k)}, j_\beta^{(k)}\rangle T_{a_1,...,a_r}^{m_1, ..., m_r}(j_1^{(k)}, ..., \widehat{j_\alpha^{(k)}}, ..., \widehat{j_\beta^{(k)}}, ..., j_{2t-2}^{(k)})
    \end{align*}
    Using the fact 
    \begin{align*}
        \sum_{k=2}^{2t}\sum_{1\leq \alpha < \beta < 2t-2} & = \sum_{k=2}^{2t}\sum_{1\leq \alpha < \beta <  k-1} + \sum_{k=2}^{2t}\sum_{1\leq \alpha < k-1 \leq \beta < 2t-2} + \sum_{k=2}^{2t}\sum_{k-1\leq \alpha < \beta < 2t-2}\\
        & = \sum_{k=2}^{2t}\sum_{1\leq \alpha < \beta \leq  k-2} + \sum_{k=2}^{2t}\sum_{1\leq \alpha \leq k-2}\sum_{k-1\leq \beta < 2t-2} + \sum_{k=2}^{2t}\sum_{k-1\leq \alpha < \beta < 2t-2},            
    \end{align*}
    we separate the sum into three parts that allows us to rewrite the indices in terms of $i_1, ..., i_{2t}$, then analyze them individually. 
    \begin{enumerate}
        \item For the first part, we compute as follows. 
        \begin{align*}
            & \frac 1 {t-1}\sum_{k=2}^{2t}\sum_{1\leq \alpha < \beta \leq  k-2}(-1)^k  \langle i_1, i_k\rangle \\
            & \qquad \cdot (-1)^{\alpha+\beta-1} \langle i_{\alpha+1}, i_{\beta+1}\rangle T_{a_1,...,a_r}^{m_1, ..., m_r}(i_2,..., \widehat{i_{\alpha+1}},..., \widehat{i_{\beta+1}}, ...,\widehat{i_k},..., i_{2t})\\
            = & \frac 1 {t-1}\sum_{k=2}^{2t}\sum_{2\leq \alpha < \beta \leq k-1} (-1)^k  \langle i_1, i_k\rangle\\
            & \qquad  \cdot (-1)^{\alpha+\beta-1} \langle i_{\alpha}, i_{\beta}\rangle T_{a_1,...,a_r}^{m_1, ..., m_r}(i_2,..., \widehat{i_{\alpha}},..., \widehat{i_{\beta}}, ...,\widehat{i_k},..., i_{2t})\\
            = & \frac 1 {t-1}\sum_{2\leq \alpha < \beta < k \leq 2t}(-1)^{\alpha+\beta+k-1}  \langle i_1, i_k\rangle  \langle i_{\alpha}, i_{\beta}\rangle T_{a_1,...,a_r}^{m_1, ..., m_r}(i_2,..., \widehat{i_{\alpha}},..., \widehat{i_{\beta}}, ...,\widehat{i_k},..., i_{2t})\\
            = & \frac 1 {t-1}\sum_{2\leq \alpha < \beta \leq 2t}(-1)^{\alpha+\beta-1}\langle i_{\alpha}, i_{\beta}\rangle \sum_{k=\beta+1}^{2t}(-1)^{k}  \langle i_1, i_k\rangle  T_{a_1,...,a_r}^{m_1, ..., m_r}(i_2,..., \widehat{i_{\alpha}},..., \widehat{i_{\beta}}, ...,\widehat{i_k},..., i_{2t})
        \end{align*}
        \item For the second part, we compute as follows 
        \begin{align*}
            & \frac 1 {t-1}\sum_{k=2}^{2t}\sum_{1\leq \alpha \leq k-2}\sum_{k-1\leq \beta \leq  2t-2}(-1)^k  \langle i_1, i_k\rangle \\
            & \qquad \cdot (-1)^{\alpha+\beta-1} \langle i_{\alpha+1}, i_{\beta+2}\rangle T_{a_1,...,a_r}^{m_1, ..., m_r}(i_2,..., \widehat{i_{\alpha+1}},..., \widehat{i_k},..., \widehat{i_{\beta+2}}, ...,i_{2t})\\
            = & \frac 1 {t-1}\sum_{k=2}^{2t}\sum_{2\leq \alpha \leq k-1}\sum_{k+1\leq \beta \leq 2t} (-1)^k  \langle i_1, i_k\rangle\\
            & \qquad  \cdot (-1)^{\alpha+\beta} \langle i_{\alpha}, i_{\beta}\rangle T_{a_1,...,a_r}^{m_1, ..., m_r}(i_2,..., \widehat{i_{\alpha}},..., \widehat{i_{\beta}}, ...,\widehat{i_k},..., i_{2t})\\
            = & \frac 1 {t-1}\sum_{2\leq \alpha < k < \beta \leq 2t}(-1)^{\alpha+\beta+k}  \langle i_1, i_k\rangle  \langle i_{\alpha}, i_{\beta}\rangle T_{a_1,...,a_r}^{m_1, ..., m_r}(i_2,..., \widehat{i_{\alpha}},..., \widehat{i_k},..., \widehat{i_{\beta}}, ..., i_{2t})\\
            = & \frac 1 {t-1} \sum_{2\leq \alpha < \beta \leq 2t} (-1)^{\alpha+\beta-1} \langle i_\alpha, i_\beta\rangle  \sum_{k=\alpha+1}^{\beta-1} (-1)^{k-1}\langle i_1, i_k\rangle  T_{a_1,...,a_r}^{m_1, ..., m_r}(i_2,..., \widehat{i_{\alpha}},..., \widehat{i_k},..., \widehat{i_{\beta}}, ..., i_{2t})
        \end{align*}
        \item For the third part, we compute as follows 
        \begin{align*}
            & \frac 1 {t-1}\sum_{k=2}^{2t}\sum_{k-1\leq \alpha < \beta \leq 2t-2 }(-1)^k  \langle i_1, i_k\rangle \\
            & \qquad \cdot (-1)^{\alpha+\beta-1} \langle i_{\alpha+2}, i_{\beta+2}\rangle T_{a_1,...,a_r}^{m_1, ..., m_r}(i_2,..., \widehat{i_{k}},..., \widehat{i_{\alpha+2}},..., \widehat{i_{\beta+2}}, ...,i_{2t})\\
            = & \frac 1 {t-1}\sum_{k=2}^{2t}\sum_{k+1\leq \alpha < \beta \leq 2t }(-1)^k  \langle i_1, i_k\rangle \\
            & \qquad \cdot (-1)^{\alpha+\beta-1} \langle i_{\alpha}, i_{\beta}\rangle T_{a_1,...,a_r}^{m_1, ..., m_r}(i_2,..., \widehat{i_{k}},..., \widehat{i_{\alpha}},..., \widehat{i_{\beta}}, ...,i_{2t})\\
            = & \frac 1 {t-1}\sum_{2\leq k <\alpha < \beta \leq 2t } (-1)^{\alpha+\beta+k-1}  \langle i_1, i_k\rangle \langle i_{\alpha}, i_{\beta}\rangle T_{a_1,...,a_r}^{m_1, ..., m_r}(i_2,..., \widehat{i_{k}},..., \widehat{i_{\alpha}},..., \widehat{i_{\beta}}, ...,i_{2t})\\
            = & \frac 1 {t-1}\sum_{2 \leq\alpha < \beta \leq 2t } (-1)^{\alpha+\beta-1}  \langle i_\alpha, i_\beta\rangle \sum_{k-2}^{\alpha-1} (-1)^k \langle i_1, i_k\rangle T_{a_1,...,a_r}^{m_1, ..., m_r}(i_2,..., \widehat{i_{k}},..., \widehat{i_{\alpha}},..., \widehat{i_{\beta}}, ...,i_{2t})
        \end{align*}
    \end{enumerate}
    Combining the formulas obtained in (1), (2) and (3), and noticing that in the final expression, for each fixed $\alpha, \beta$, the summation over $k$ precisely goes over all possible indices other than $\alpha, \beta$, from the definition of $T_{a_1,...,a_r}^{m_1, ..., m_r}$, we may conclude that 
    \begin{align*}
        T_{a_1,...,a_r}^{m_1, ..., m_r}(i_1, ..., i_{2t}) = \frac 1{t-1}\sum_{2\leq \alpha < \beta \leq 2t} (-1)^{\alpha+\beta-1} T_{a_1,...,a_r}^{m_1, ..., m_r}(i_1, ..., \widehat{i_\alpha}, ..., \widehat{i_\beta} ..., i_{2t})
    \end{align*}
    Multiply both sides by $t-1$, add up with another copy of $T_{a_1,...,a_r}^{m_1, ..., m_r}(i_1, ..., i_{2t})$ in terms of recursion of $i_1$, we may conclude that 
    \begin{align*}
        t\cdot T_{a_1,...,a_r}^{m_1, ..., m_r}(i_1, ..., i_{2t}) = \sum_{1\leq \alpha < \beta \leq 2t} (-1)^{\alpha+\beta-1} T_{a_1,...,a_r}^{m_1, ..., m_r}(i_1, ..., \widehat{i_\alpha}, ..., \widehat{i_\beta} ..., i_{2t}),
    \end{align*}
    which is what we want. 
\end{proof}

\begin{rema}
    Indeed, the total contraction number admits a geometric interpretation: Introduce a graph with vertices labeled by $i_1, ..., i_{2t}$ and with $t$ edges. Each pair of vertices linked by exactly one edge. We may associate a number to such a graph by the following rule: 
    \begin{enumerate}
        \item Each edge linking the indices $i$ and $j$ contributes a factor of $[i,j]$ (as the notation in the proof of Proposition \ref{Contraction-Number-t-recursion}); 
        \item Each intersection of two edges contributes to a $(-1)$ factor. 
    \end{enumerate}
    For example, here is the association of all possible graphs with 4 vertices: 
    \begin{center}
    \raisebox{-.3 \height}{\begin{tikzpicture}[line cap=round,line join=round,>=triangle 45,x=1cm,y=1cm, scale = 0.5]
        \draw [shift={(1.5,0)},line width=2pt]  plot[domain=0:3.141592653589793,variable=\t]({1*0.5*cos(\t r)+0*0.5*sin(\t r)},{0*0.5*cos(\t r)+1*0.5*sin(\t r)});
        \draw [shift={(3.5,0)},line width=2pt]  plot[domain=0:3.141592653589793,variable=\t]({1*0.5*cos(\t r)+0*0.5*sin(\t r)},{0*0.5*cos(\t r)+1*0.5*sin(\t r)});
            \begin{scriptsize}
                \draw [fill=ffffff] (1,0) circle (2.5pt);
                \draw [fill=ffffff] (2,0) circle (2.5pt);
                \draw [fill=ffffff] (3,0) circle (2.5pt);
            \draw [fill=ffffff] (4,0) circle (2.5pt);
            \end{scriptsize}
        \end{tikzpicture}} $\mapsto \langle i_1, i_2\rangle\langle i_3,i_4\rangle ,$ \hspace{0.5 cm} \raisebox{-.3\height}{\begin{tikzpicture}[line cap=round,line join=round,>=triangle 45,x=1cm,y=1cm, scale = 0.5]
        \draw [shift={(2,0)},line width=2pt]  plot[domain=0:3.141592653589793,variable=\t]({1*1*cos(\t r)+0*1*sin(\t r)},{0*1*cos(\t r)+1*1*sin(\t r)});
        \draw [shift={(3,0)},line width=2pt]  plot[domain=0:3.141592653589793,variable=\t]({1*1*cos(\t r)+0*1*sin(\t r)},{0*1*cos(\t r)+1*1*sin(\t r)});
        \begin{scriptsize}
        \draw [fill=ffffff] (1,0) circle (2.5pt);
        \draw [fill=ffffff] (2,0) circle (2.5pt);
        \draw [fill=ffffff] (3,0) circle (2.5pt);
        \draw [fill=ffffff] (4,0) circle (2.5pt);
        \end{scriptsize}
        \end{tikzpicture}} $\mapsto -\langle i_1, i_3\rangle \langle i_2, i_4\rangle$, \hspace{0.5 cm} \raisebox{-.3\height}{\begin{tikzpicture}[line cap=round,line join=round,>=triangle 45,x=1cm,y=1cm, scale = 0.5]
        \draw [shift={(2.5,0)},line width=2pt]  plot[domain=0:3.141592653589793,variable=\t]({1*1.5*cos(\t r)+0*1.5*sin(\t r)},{0*1.5*cos(\t r)+1*1.5*sin(\t r)});
        \draw [shift={(2.5,0)},line width=2pt]  plot[domain=0:3.141592653589793,variable=\t]({1*0.5*cos(\t r)+0*0.5*sin(\t r)},{0*0.5*cos(\t r)+1*0.5*sin(\t r)});
        \begin{scriptsize}
        \draw [fill=ffffff] (1,0) circle (2.5pt);
        \draw [fill=ffffff] (2,0) circle (2.5pt);
        \draw [fill=ffffff] (3,0) circle (2.5pt);
        \draw [fill=ffffff] (4,0) circle (2.5pt);
        \end{scriptsize}
        \end{tikzpicture}} $\mapsto \langle i_1, i_4\rangle\langle i_2, i_3\rangle. $
    \end{center}
    Then the total contraction number $T_{a_1, ..., a_r}^{m_1, ..., m_r}(i_1,..., i_{2t})$ is the sum of the associated numbers over all the possible graphs with $2t$ indices (where there are $(2t-1)!!$ choices). 
\end{rema}

\begin{prop}
For every positive integer $t$,
\begin{align*}
    \frac{\Delta(x)^t}{t!} a_1(-m_1&-1/2)\cdots a_r(-m_r-1/2)\one\\
    =\sum_{1\leq i_1 < \cdots < i_{2t} \leq r}&(-1)^{i_1+\cdots + i_{2t}}x^{-m_{i_1} - \cdots - m_{i_t} - t} T_{a_1,...,a_r}^{m_1,...,m_r}(i_1, ..., i_{2t})
    \\& \cdot a_1(-m_1-1/2)\cdots \reallywidehat{a_{i_1}(-m_{i_1}-1/2)}\cdots \reallywidehat{a_{i_{2t}}(-m_{i_{2t}}-1/2)} \cdots a_{r}(-m_r-1/2)\one
\end{align*}
\end{prop}

\begin{proof}
    We prove by induction. With Notation \ref{Index-Notation-1} and from Proposition \ref{Delta(x)-Properties}, 
    $$\Delta(x) S_a(i_1, ..., i_k) = \sum_{1\leq p < q \leq k} (-1)^{p+q}\langle i_p, i_q\rangle x^{-m_{i_p}-m_{i_q}-1}S_{a_1,... a_r}^{m_1, ..., m_r}(i_1, ..., \widehat{i_p}, ..., \widehat{i_q}, ..., i_k) $$
    So the base case is proved. Assume the conclusion holds for all smaller $t$, then
    \begin{align*}
        & \frac{\Delta(x)^{t-1}}{(t-1)!} S_{a_1,... a_r}^{m_1, ..., m_r}(1, ..., r)\\ 
        = \ & \sum_{1\leq i_1 < \cdots < i_{2t-2}\leq r} (-1)^{i_1+\cdots + i_{2t-2}} x^{-m_{i_1}-\cdots -m_{i_{2t-2}}-(t-1)} \\
        & \qquad \qquad \qquad  \cdot T_{a_1, ..., a_r}^{m_1, ..., m_r} (i_1, ..., i_{2t-2})  S_{a_1, ..., a_r}^{m_1, ..., m_r}(1, ..., \widehat{i_1}, ..., \widehat{i_{2t-2}}, ..., r)\\
        = \ & \sum_{1\leq i_1 < \cdots < i_{2t-2}\leq r} (-1)^{i_1+\cdots + i_{2t-2}} x^{-m_{i_1}-\cdots -m_{i_{2t-2}}-(t-1)} \\
        & \qquad \qquad \qquad  \cdot T_{a_1, ..., a_r}^{m_1, ..., m_r} (i_1, ..., i_{2t-2})  S_{a_1, ..., a_r}^{m_1, ..., m_r}(j_1, ..., j_{r-2t+2})
    \end{align*}
    Herr $1\leq j_1< ...< j_{r-2t+2}\leq r$ be the complement of $i_1, ..., i_{2t}$, i.e., 
    $$j_\alpha =\left\{\begin{array}{ll}
    \alpha & \text{ if }1 \leq \alpha \leq i_1-1\\
    \alpha+1 & \text{ if } i_1 \leq \alpha \leq i_2-2 \\
    \cdots & \cdots \\
    \alpha + \beta & \text{ if } i_\beta - \beta + 1 \leq p \leq i_{\beta+1} - (\beta+1)\\
    \cdots & \cdots \\
    \alpha + 2t-2 & \text{ if } i_{2t-2} -(2t-2)+ 1 \leq \alpha \leq r-(2t-2)
    \end{array}\right.$$
    Now we apply $\Delta(x)$, so as to get 
    \begin{align*}
        & \frac{\Delta(x)^{t}}{(t-1)!} S_{a_1,... a_r}^{m_1, ..., m_r}(1, ..., r)\\ 
        = \ & \sum_{1\leq i_1 < \cdots < i_{2t-2}\leq r} (-1)^{i_1+\cdots + i_{2t-2}} x^{-m_{i_1}-\cdots -m_{i_{2t-2}}-(t-1)} T_{a_1, ..., a_r}^{m_1, ..., m_r} (i_1, ..., i_{2t-2})\\
        & \qquad \qquad \qquad  \cdot \sum_{1\leq p < q \leq r-2t+2} (-1)^{p+q} \langle j_p, j_q\rangle x^{-m_{p}-m_q - 1} S_{a_1, ..., a_r}^{m_1, ..., m_r}(j_1, ..., \widehat{j_p}, ..., \widehat{j_q}, ..., , j_{r-2t+2})
    \end{align*}
    We would use the same idea from the proof of Proposition \ref{Contraction-Number-t-recursion} to interchange the order of summation. Note that for fixed $1 \leq i_1 < \cdots < i_{2t-2} \leq r$, 
    \begin{align*}
         \sum_{1\leq p< q \leq r-2t+2}= \ & \sum_{1\leq p< q \leq i_1 -1 } +  \sum_{\beta = 1}^{2t-3}\sum_{i_\beta - \beta +1 \leq p< q \leq i_{\beta+1} -(\beta+1) } +  \sum_{i_{2t-2}-(2t-2)+1\leq p<q\leq r-2t+2} \\
        & + \sum_{1\leq p \leq i_1-1} \sum_{\beta_2=1}^{2t-3}\sum_{i_{\beta_2}-\beta_2 + 1 \leq q \leq i_{\beta_2+1}-(\beta_2+1)} + \sum_{1\leq p \leq i_1-1} \sum_{i_{2t-2}-(2t-2) + 1 \leq q \leq i_{\beta_2+1}-(\beta_2+1)}\\
        & + \sum_{1\leq \beta_1 < \beta_2 \leq 2t-3} \sum_{i_{\beta_1}-\beta_1+1\leq p \leq i_{\beta_1+1}-(\beta_1+1)} \sum_{i_{\beta_2}-\beta_2+1\leq q \leq i_{\beta_2+1}-(\beta_2+1)} \\
        & + \sum_{\beta_1=1}^{2t-3} \sum_{i_{\beta_1}-\beta_1 + 1 \leq p \leq i_{\beta_1+1}-(\beta_1+1)} \sum_{i_{2t-2}-(2t-2)+1\leq q \leq r-2t+2}. 
    \end{align*} 
    By a similar though more complicated process, we will obtain that 
    \begin{align*}
        \frac{\Delta(x)^t}{(t-1)!} =\ & \sum_{1\leq i_1 < \cdots < i_{2t}\leq r}(-1)^{i_1 + \cdots + i_{2t}}x^{-m_{i_1}-\cdots - m_{i_{2t}}-t}S_{a_1, ..., a_r}^{m_1, ..., m_r}(1, ..., \widehat{i_1},..., \widehat{i_{2t}}, ..., r)\\
        & \qquad \qquad \qquad \cdot \sum_{1\leq \alpha < \beta \leq 2t}(-1)^{\alpha+\beta-1}\langle i_\alpha, i_\beta\rangle T_{a_1, ..., a_r}^{m_1, ...,m_r} (i_1, ..., \widehat{i_\alpha}, ..., \widehat{i_\beta}, ..., i_{2t}) \\
        = \ &  \sum_{1\leq i_1 < \cdots < i_{2t}\leq r}(-1)^{i_1 + \cdots + i_{2t}}x^{-m_{i_1}-\cdots - m_{i_{2t}}-t}S_{a_1, ..., a_r}^{m_1, ..., m_r}(1, ..., \widehat{i_1},..., \widehat{i_{2t}}, ..., r)\\
        & \qquad \qquad \qquad  \cdot t\cdot  T_{a_1, ..., a_r}^{m_1, ...,m_r} (i_1, ..., i_{2t}) 
    \end{align*}
    The conclusion then follows from dividing $t$ on both sides. 
\end{proof}

With these results, we have an expression of the $\exp(\Delta(x))$-action on every basis element of $V$. 
\begin{cor}
    For $a_1, ..., a_r\in \h, m_1, ..., m_r\in \N, $
    \begin{align*}
        & \exp(\Delta(x)) a_1(-m_1-1/2)\cdots a_r(-m_r-1/2)\one \\
        = \ & \sum_{t=0}^\infty \sum_{1\leq i_1 < \cdots < 2t \leq r} (-1)^{i_1 + \cdots + i_{2t}} T_{a_1, .. a_r}^{m_1, ..., m_r}(i_1, ..., i_{2t})\\
        & \qquad \qquad \qquad \cdot a_1(-m_1-1/2)\cdots \reallywidehat{a_{i_1}(-m_{i_1}-1/2)}\cdots \reallywidehat{a_{i_{2t}}(-m_{i_{2t}}-1/2)} \cdots a_{r}(-m_r-1/2)\one. 
    \end{align*}
\end{cor}

\subsection{Commutator of $\exp(\Delta(x))$ and $a^{(m)}(x)^-$}

\begin{prop}\label{Exp-Delta-Neg-Comm}
    For $a\in \h, m\in \N$, 
    \begin{align}
        [\exp(\Delta(y)), a^{(m)}(x)^-] &= \sum_{\alpha\geq 0}\sum_{\beta\geq 0}C_{\beta\alpha}a(\beta + 1/2)y^{-\beta-\alpha-1}(x^\alpha)^{(m)}\label{Exp-Delta-Neg-Comm-1} \exp(\Delta(y))
    \end{align}
\end{prop}

\begin{proof}
    Note that from Lemma \ref{Delta(x)-Properties}, $[\Delta(y), a(-n-1/2)]$ is a series consisting of positive modes that are anticommutative with positvie modes. Note also that $\Delta(y)$ itself is a product of two positive modes. Therefore, we see that $\Delta(y)$ commutes with $[\Delta(y), a(-n-1/2)]$, i.e., 
    $$\Delta(y)[\Delta(y), a(-n-1/2)] = [\Delta(y), a(-n-1/2)]\Delta(y)$$
    Clearly, since $\Delta(y)$ commutes with itself, for every $r\geq 1$, by an easy argument of induction, we have
    \begin{align*}
         [\Delta(y)^r, a(-n-1/2)] &= \Delta(y)[\Delta(y)^{r-1}, a(-n-1/2)] + [\Delta(y), a(-n-1/2)]\Delta(y)^{r-1}\\
        &  = r [\Delta(y), a(-n-1/2)]\Delta(y)^{r-1}. 
    \end{align*}
    So, 
    \begin{align*}
         \exp(\Delta(y)) a(-n-1/2) & = \sum_{r=0}^\infty \frac 1 {r!} \Delta(y)^r a(-n-1/2) \\
        & = \sum_{r=0}^\infty \frac 1 {r!}  \left(a(-n-1/2) \Delta(y)^r + r [\Delta(y), a(-n-1/2)] \Delta(y)^{r-1} \right)\\
        &  = a(-n-1/2) \exp(\Delta(y)) + [\Delta(y), a(-n-1/2)] \sum_{r=1}^\infty \frac 1 {(r-1)!} \Delta(y)^{r-1}\\
        &  = {\bigg (}a(-n-1/2) +[\Delta(y), a(-n-1/2)] {\bigg )}\exp(\Delta(y))
    \end{align*}
Therefore, 
\begin{align*}
    [\exp(\Delta(y)), a^{(m)}(x)^-] & =\sum_{\alpha\geq 0} \binom \alpha m  [\Delta(y), a(-\alpha-1/2)] x^{\alpha-m} \exp(\Delta(y))\\
    & =  \sum_{\alpha\geq 0} \binom{\alpha} m \sum_{\beta\geq 0} C_{\beta\alpha} a(\beta+1/2)y^{-\beta-\alpha-1}  x^{\alpha-m}\exp(\Delta(y))\\
    & =  \sum_{\alpha\geq 0} \sum_{\beta\geq 0} C_{\beta\alpha} a(\beta+1/2)y^{-\beta-\alpha-1}  (x^{\alpha})^{(m)}\exp(\Delta(y))
\end{align*}
So we proved (\ref{Exp-Delta-Neg-Comm-1}). 
\end{proof}

\noindent {\small \sc Department of Mathematics, Rutgers University, Piscataway, NJ  08854, USA}

\noindent {\em E-mail address}: fiorda@scarletmail.rutgers.edu,

\noindent {\small \sc Department of Mathematics, University of Denver, Denver, CO  80210, USA}

\noindent {\em E-mail address}: fei.qi@du.edu | fei.qi.math.phys@gmail.com


\begin{thebibliography}{KWAK2}








\bibitem[B]{B} Richard E. Borcherds, Vertex algebras, Kac-Moody algebras, and the
Monster, \textit{Proc. Natl. Acad. Sci. USA} \textbf{83} (1986), 3068--3071.

\bibitem[BPZ]{BPZ} Alexander A. Belavin, Alexander M. Polyakov and Alexander B. Zamolodchikov, Infinite conformal symmetry
in two-dimensional quantum field theory, \textit{Nucl. Phys.} \textbf{B241} (1984) 333-380 





\bibitem[DPZ]{DPZ} Paolo Di Vecchia, Jens Lyng Petersen, Hai Bin Zhang, $N=2$ extended superconformal theories in two dimensions, \textit{Physics Letters B}, \textbf{162}, Issues 4--6, Pages 327--332


\bibitem[FFR]{FFR} Alex J. Feingold, Igor B. Frenkel, John F. X. Ries, \textit{Spinor Construction of Vertex Operator Algebras, Triality, and $E_8^{(1)}$}, Contemporary Mathematics, Vol. 121, American Mathematical Society, Providence, RI, 1991. 


\bibitem[FLM]{FLM} Igor B. Frenkel, James I. Lepowsky and Arne  Meurman, {\it Vertex operator algebra and the monster},  Pure and Applied Mathematics, 134, Academic Press, New York, 1988.

\bibitem[G]{G-PBW} Paul Garrett, Poincare-Birkhoff-Witten Theorem, https://www-users.cse.umn.edu/~garrett/m/algebra/pbw.pdf




\bibitem[H1]{H-MOSVA}
Yi-Zhi Huang, Meromorphic open string vertex algebras, \textit{J. Math. Phys.} \textbf{54} (2013), 051702. 

\bibitem[H2]{H-MOSVA-Riemann}
Yi-Zhi Huang, Meromorphic open-string vertex
algebras and Riemannian manifolds, arXiv:1205.2977.

\bibitem[H3]{H-Two-Constructions} Yi-Zhi Huang, Two constructions of grading-restricted vertex (super)algebras, \textit{J. Pure Appl. Alg.} \textbf{220} (2016), 3628-3649.



\bibitem[HK]{HK-OSVA} 
Yi-Zhi Huang, Liang Kong, Open-string vertex algebras, tensor categories and operads, {\it Comm. Math. Phys.},  \textbf{250} (2004), 433--471





\bibitem[KW]{KW} Victor Kac and Weiqiang Wang, Vertex operator superalgebras and their representations, \textit{Mathematical aspects of conformal and topological field theories and quantum
groups}, Cntemporary Mathematics, Vol. 175, (1994) 161--191.


\bibitem[LL]{LL} 
James Leposwky, Haisheng Li, \textit{Introduction to vertex operator algebras
and their representations}, Progress in Mathematics, 227, Birkh\"auser, Boston, 2004. 


 \bibitem[MS]{MS}
Greg Moore, Nathan Seiberg, Classical and quantum conformal field theory, \textit{Comm. Math. Phys.} \textbf{123} (1989), 177--254.

\bibitem[NS]{NS} Andr\'e Neveu and John H. Schwarz: Factorizable Dual Model of Pions, \textit{Nucl. Phys. B} \textbf{31} (1971) 86--112.






\bibitem[Q1]{Q-Mod} Fei Qi, On modules for meromorphic open-string vertex algebras, \textit{J. Math. Phys.} \textbf{60}, 031701 (2019)

\bibitem[Q2]{Q-2d-space-form} Fei Qi, Meromorphic open-string vertex algebras and modules over two-dimensional orientable space forms, \textit{Lett. Math. Phys. }\textbf{111} Article 27 (2021), 1--54. 

\bibitem[Q3]{Q-Cov} Fei Qi, Covariant derivatives of eigenfunctions along parallel tensors over space forms and a conjecture motivated by the vertex algebraic structure, \textit{J. Noncommut. Geom.} \textbf{2}, (2022),  717--759


\bibitem[Q4]{Q-Fermion-2} Fei Qi, Fermionic construction of the $\frac \Z 2$-graded meromorphic open-string vertex algebra and its $\Z_2$-twisted module, II, to appear. 


\bibitem[T]{T} Haruo Tsukada, Vertex operator superalgebras, \textit{Comm. Alg.}, \textbf{18}:7 (1990), 2249--2274. 


\end{thebibliography}
\end{document}